%% file: averaging_HJ_2019.tex
\documentclass[12pt, a4paper, final]{ amsart}
  
\usepackage{amsmath, amssymb, amsthm, amsrefs}
\usepackage{empheq, comment}

\usepackage{caption}
\usepackage{color}

\usepackage{geometry}
\allowdisplaybreaks[2]

\numberwithin{equation}{section}

\theoremstyle{plain}
\newtheorem{thm}{Theorem}[section]
\newtheorem*{thm*}{Theorem}

\newtheorem{cor}[thm]{Corollary}
\newtheorem{lem}[thm]{Lemma}
\newtheorem{prop}[thm]{Proposition}

\theoremstyle{definition}
\newtheorem{defn}{Definition}[section]
\newtheorem{exam}{Example}[section]

\DeclareMathOperator{\USC}{USC}
\DeclareMathOperator{\LSC}{LSC}

\def\ga{\alpha}
\def\gb{\beta}

\def\gd{\delta}

\def\gk{\kappa} 
\def\gl{\lambda}

\def\gO{\Omega}
                
\def\gs{\sigma}

\def\gth{\theta}

\def\gz{\zeta}

\def\ep{\varepsilon}    
\def\gam{\gamma}

\def\tim{\times}

\def\sbs{\subset}

\def\bproof{\begin{proof}}\def\eproof{\end{proof}}
\def\beq{\begin{equation}}\def\eeq{\end{equation}}
\def\bcases{\begin{cases}}\def\ecases{\end{cases}}

\def\pl{\partial}

\def\ol{\overline}

\def\cfr{\cfrac}

\def\R{\mathbb R}
\def\N{\mathbb N}

\def\cF{\mathcal F}
\def\cI{\mathcal I}

\def\cS{\mathcal S}

\def\For{\text{ for }}
\def\In{\text{ in }}
\def\If{\text{ if }}
\def\On{\text{ on }}

\def\bye{\end{document}}

\def\0{\{ 0 \}}

\def\ON{\text{ on }}
\def\IN{\text{ in }}
\def\bald{\begin{aligned}}
\def\eald{\end{aligned}}
\def\map{\operatorname{ :\, }}
\def\dist{\operatorname{dist}}

\title{Averaging of Hamilton-Jacobi equations over Hamiltonian flows}

\author[H. Ishii]{Hitoshi Ishii$^*$}
\address[H. Ishii]{Institute for Mathematics and Computer Science, Tsuda  University,
 2-1-1 Tsuda, Kodaira, Tokyo 187-8577 Japan.}
\email{hitoshi.ishii@waseda.jp}
\thanks{${}^*$ Corresponding author}

\author[T. Kumagai]{Taiga Kumagai}
\address[T. Kumagai]{Academic Support Center, Kogakuin Univesity,
2665-1 Nakano-machi, Hachioji-shi, Tokyo 192-0015 Japan.}
\email{kumatai13@gmail.com}

\date{}


\begin{document}
\maketitle

\begin{abstract}
We study the asymptotic behavior of solutions to the Dirichlet problem for Hamilton-Jacobi equations with  large drift terms, where the drift terms are given by the 
Hamiltonian vector fields of Hamiltonian $H$. This is an attempt to understand the averaging effect for fully nonlinear degenerate elliptic equations. In this work, we restrict ourselves to the case of Hamilton-Jacobi equations. 
The second author has already established averaging results for Hamilton-Jacobi equations with convex Hamiltonians ($G$ below) under the classical formulation of the Dirichlet condition. Here we treat the Dirichlet condition in the viscosity sense, and establish an averaging result for Hamilton-Jacobi equations with relatively general Hamiltonian $G$.   
\end{abstract}

\tableofcontents


\advance \baselineskip by 3pt

\section{Introduction}

In this paper, we consider the Dirichlet problem for the Hamilton-Jacobi equation
\begin{align}  
\gl u^\ep - \frac{1}{\ep}\,b \cdot Du^\ep + G(x, Du^\ep) = 0 \ \ \  &\In \gO, \tag{$\mathrm{HJ}^\ep$} \label{E: epHJ} \\ 
u^\ep = g \ \ \  &\On \pl \gO. \tag{$\mathrm{BC}^\ep$} \label{E: epBC}
\end{align}
Here $\gl > 0$ and $\ep>0$ are constants, $\gO \sbs \R^2$ is an open and bounded set, $u^\ep : \ol\gO \to \R$ is the unknown function, and 
$G : \ol\gO \tim \R^2 \to \R$ and $g : \pl\gO \to \R$ are given functions. 

Our primary purpose is to investigate the behavior, as $\ep\to 0+$, of the solution $u^\ep$ of \eqref{E: epHJ} and \eqref{E: epBC}. 

\def\prob{\eqref{E: epHJ} and \eqref{E: epBC}}

In the problem \prob, our choice of the domain $\gO$ and the vector field $b$ features as follows:  
we are given a function $H : \R^2 \to \R$, called a Hamiltonian, that has 
the properties (H1)--(H3) described below. 
Let $N$ be an integer such that $N\geq 2$. 
Set $\cI_0 := \{ 0, 1, \ldots, N-1 \}$ and $\cI_1 := \{ 1, \ldots, N-1 \}$. 
 
\begin{itemize}
\item[(H1)] $H \in C^2(\R^2)$ and $\lim_{|z| \to \infty} H(z) = \infty$.

\item[(H2)] $H$ has exactly $N$ critical points $z_i \in \R^2$, with $i \in \cI_0$, and attains a local minimum at every $z_i$, with $i \in \cI_1$. 
Moreover $z_0 = 0$ and $H(0) = 0$. 

\item[(H3)] There exist constants $m \geq 0$, $n > 0$, $A_1 > 0$, $A_2 > 0$ and a neighborhood $V \sbs \R^2$ of $0$ such that $n < m+2$ and
\begin{equation*}
|H_{x_i x_j}(x)| \leq A_1|x|^m \ \ \  \text{ for all } x \in V \text{ and } i,j \in \{1, 2\}, 
\end{equation*}
and
\begin{equation*} 
A_2|x|^n \leq |DH(x)| \ \ \  \text{ for all } x \in V.
\end{equation*}
\end{itemize}

The geometry of $H$ are stated as follows (see also \cite{Ku17}).  The set 
$D_0 = \{ x \in \R^2 \mid H(x) > 0 \}$ is open and connected, 
and the open set $\{ x \in \R^2 \mid H(x) < 0 \}$ has exactly $N-1$ connected components $D_i$, with $i \in \cI_1$, such that $z_i \in D_i$ (see Figure 1).
Furthermore, it follows that $\pl D_0 := \{ x \in \R^2 \mid H(x) = 0 \}$,  
$\pl D_0 = \bigcup_{i \in \cI_1} \pl D_i$, and $\pl D_i \cap \pl D_j = \{ 0 \}$ if 
$i,j\in\cI_1$ and $i \not= j$.
\begin{figure}[h]
\centering
\input{domain.tex}
\caption{$N=6$}
\end{figure}

We choose $h_i \in \R$, with $i \in \cI_0$, so that
\begin{equation*}
h_0 > 0 \ \ \  \text{ and } \ \ \  H(z_i) < h_i < 0 \ \ \  \For i \in \cI_1,
\end{equation*}
and 
define 
\begin{equation*}
\gO_0 = \{ x \in D_0 \mid H(x) < h_0 \}, \ \ \  \ \ \  \gO_i = \{ x \in D_i \mid H(x) > h_i \} \ \ \  \For i \in \cI_1,
\end{equation*}
and
\begin{equation*}
\pl_i\gO = \{ x \in \ol\gO_i \mid H(x) = h_i \} \ \ \  \For i \in \cI_0.
\end{equation*}
Finally, the set $\gO$ is given by  
\begin{equation*}
\gO = \{ x \in \R^2 \mid H(x) = 0 \} \cup \bigcup_{i \in \cI_0} \gO_i,
\end{equation*}
and the drift vector $b : \R^2 \to \R^2$ is given by the Hamiltonian vector field of $H$, that is,
\begin{equation*} 
b = (H_{x_2}, -H_{x_1}).
\end{equation*}
Note that 
\begin{equation*}
\pl\gO = \bigcup_{i \in \cI_0} \pl_i\gO.
\end{equation*}

Our primary interest in this work is to generalize fully the averaging results obtained by Freidlin-Wentzell \cite{FW94} and Ishii-Souganidis \cite{IS12} for stochastic processes to those for controlled stochastic processes. 
The analysis of the averaging of stochastic processes can be phrased, in terms of partial differential equations, as the study of the asymptotic behavior of solutions to 
linear second-order elliptic partial differential equations, 
with the large Hamiltonian drift term $-b \cdot Du^\ep/\ep$, 
while for controlled stochastic processes, 
fully nonlinear second-order degenerate elliptic equations, of the form
\begin{equation} \label{E: FN}
\gl u^\ep - \frac{1}{\ep}\,b \cdot Du^\ep + G(x, Du^\ep, D^2u^\ep) = 0 \ \ \  \In \gO,
\end{equation}
take over the role of linear elliptic equations. 

However, by the technical reasons, we restrict ourselves to the case where the function $G$ of $(x,u,Du,D^2u)$ in \eqref{E: FN} does not depend on $D^2u$. That is, we treat here the first-order equation \eqref{E: epHJ}. In other words, we deal with 
deterministic control or differential games processes.  
The second author  has already studied the asymptotic problem for such deterministic processes by analyzing 
\eqref{E: epHJ} and \eqref{E: epBC}.

A crucial difference of this work from \cites{Ku16, Ku17} is that $G$ is not anymore convex so that the results cover the differential games processes.
Another critical point here is that we treat the Dirichlet boundary condition 
in the viscosity sense, which makes the statement of our results transparent.

There are two difficulties to be dealt with here beyond those in \cite{Ku16, Ku17}. 
One is that the optimal control interpretation is not available anymore of the problem, and the second is how to deal with the boundary layer and to determine 
the effective boundary data. 
The bottom line to solve these difficulties is that the perturbed Hamiltonian $-\ep^{-1}b(x) \cdot p + G(x, p)$ is coercive in the direction of $DH(x)$ although it is not coercive in the other directions when $\ep$ is very small.

Our result is stated in Theorem \ref{T: main}, which claims that the effective problem is identified with the Dirichlet problem for a Hamilton-Jacobi equation on a graph.
Indeed, the large Hamiltonian drift term, as $\ep \to 0+$, makes $u^\ep$ nearly constant along the level sets of $H$.
If we identify every $h$-level set of $H$ in $\gO_i$ with a point $h$ in the intervals
\begin{equation*}
J_0 = (0, \, h_0) \ \ \  \text{ and } \ \ \  J_i = (h_i, \, 0) \ \ \  \For i \in \cI_1
\end{equation*}
and the zero level set of $H$ with point $0$ connecting all the intervals $J_i$, then we obtain a graph consisting of one node $0$ and $N$ edges $J_i$.
These suggest that the limit problem should be posed naturally and effectively on the graph. 

Various definitions of viscosity solutions on graphs have been introduced in the literature, and we refer for these to \cite{ACCT13, IMZ13, IM17, LS16, LS17}, although those cannot be adopted to our effective problem. Our effective Hamiltonians in the edges are not well-defined  at the node and their coercivities break down near the node. 
In our result, we identify the limit function of $u^\ep$ with a maximal continuous viscosity solution of the effective problem posed on the graph. 
We also refer to \cite{AT15, LS16, GIM15} for asymptotic problems related to ours, in which Hamilton-Jacobi equations on graphs appear as effective problems.

This paper is organized as follows.
In the next section, we give some assumptions on $G$ and a basic existence result for \eqref{E: epHJ} and \eqref{E: epBC} as well as a typical example of $G$ satisfying the assumptions. 
In Section \ref{S: main result}, we present the main results. Section \ref{S: effective prob}
makes fundamental observations concerning the effective problem in the edges. 
Section \ref{S: proof main} outlines the proof of the main theorem based on three propositions 
and proves one of these propositions. The other two propositions are shown 
in Sections \ref{S: visco} and \ref{S: maximal}, respectively. In the appendix a basic proposition is 
presented together with its proof. 
\medskip

\noindent{\bf Notation:} For a function $f\map X\to \R^m$, we write 
\ $\|f\|_{\infty}=\|f\|_{\infty,X}:=\sup\{|f(x)|\mid x\in X\}$.  For $r_1,\,r_2\in\R$, we write $r_1\wedge r_2:=\min\{r_1,\,r_2\}$ and $r_1\vee r_2:=\max\{r_1,\,r_2\}$. 

\section*{Acknowledgements} The work of the first author was partially supported by the JSPS grants: KAKENHI \#16H03948, \#18H00833. He also thanks to the Department of Mathematics at the Sapienza University of Rome,  the School of Mathematical Sciences at Fudan University, the Department of 
Mathematics at the Pontifical Catholic University of Rio de Janeiro,  
for their financial support and warm hospitality, while his visits. 
The work of the second author was mostly done while he was a member of the Faculty of Education and Integrated Arts and Science at Waseda University (April 2017--May 2018) and of the Department of Mathematics at Tokyo Institute Technology (June 2018--November 2019).

\section{The problem \eqref{E: epHJ} and \eqref{E: epBC}} 

This section concerns the problem \eqref{E: epHJ} and \eqref{E: epBC}.
We set
\begin{equation*}
\bar h = \min_{i \in \cI_0} |h_i| \ \ \  \text{ and } \ \ \  \gO(s) = \{ x \in \R^2 \mid |H(x)| < s \} \ \ \  \For s \in (0, \, \bar h), 
\end{equation*}
and denote the closure of $\gO(s)$ by $\ol \gO(s)$.

We need the following assumptions.

\begin{itemize}
\item[(G1)] $G \in C(\ol\gO \tim \R^2)$ and $g \in C(\pl\gO)$. 

\item[(G2)] There exists a continuous nondecreasing function $m_1 : [0, \, \infty) \to [0, \, \infty)$ satisfying $m_1(0) = 0$ such that
\begin{equation*}
|G(x, p) - G(y, p)| \leq m_1(|x-y|(1+|p|)) \ \ \  \text{for all $x, y \in \ol\gO$ \ and \ $p \in \R^2$.}
\end{equation*}

\item[(G3)]\label{item:G3}There exists a continuous nondecreasing function $m_2 : [0, \, \infty) \to [0, \, \infty)$ satisfying $m_2(0) = 0$ such that
\begin{equation*}
|G(x, p) - G(x, q)| \leq m_2(|p-q|)\ \ \  \text{ for all $x \in \ol\gO$ \ and \ $p, q \in \R^2$.}
\end{equation*}

\item[(G4)] There exists $\gam \in (0, \, \bar h)$ such that, for each $x \in \gO(\gam) \setminus c_0(0)$, the function $\R \ni q \mapsto G(x, qDH(x))$ is convex.

\item[(G5)] There exist $\nu > 0$ and $M > 0$ such that
\begin{equation*}
G(x, p) \geq \nu |p| - M \ \ \  \text{ for all } (x, p) \in \ol\gO \tim \R^2.
\end{equation*}
\end{itemize}

As already mentioned in the introduction, in this paper, we deal with solutions satisfying Dirichlet boundary conditions in the sense of viscosity solutions.
We now recall the definition (see e.g. \cite{CIL92, I89})  of viscosity solutions to \eqref{E: epHJ} as well as those to \eqref{E: epHJ} and \eqref{E: epBC}. 

In what follows, we always assume (G1).

\begin{defn}
A function $u : \gO \to \R$ is called a viscosity subsolution {\rm (}resp., supersolution{\rm )} of \eqref{E: epHJ} if $u$ is locally bounded in $\gO$ and, for any $\phi \in C^1(\gO)$ and $z \in \gO$ such that $u^*-\phi$ attains a local maximum {\rm (}resp., $u_*-\phi$ attains a local minimum{\rm )} at $z$,
\begin{gather*}
\gl u^*(z) - \ep^{-1}b(z) \cdot D\phi(z) + G(z, D\phi(z)) \leq 0 \\
(\text{resp., } \gl u_*(z) - \ep^{-1}b(z) \cdot D\phi(z) + G(z, D\phi(z)) \geq 0),
\end{gather*}
where $u^*$ and $u_*$ denote, respectively, the upper and lower semicontinuous envelope of $u$.
A function $u : \gO \to \R$ is called a viscosity solution of \eqref{E: epHJ} if $u$ is both a viscosity sub- and supersolution of \eqref{E: epHJ}.
\end{defn}

\begin{defn}
A function $u : \ol\gO \to \R$ is called a viscosity subsolution {\rm (}resp., supersolution{\rm )} of \eqref{E: epHJ} and \eqref{E: epBC} if $u$ is bounded on $\ol\gO$ and the following two conditions {\rm (i), (ii)} hold{\rm :} {\rm (i)} $u$ is a viscosity subsolution {\rm (}resp., supersolution{\rm )} of \eqref{E: epHJ}, {\rm (ii)} for any $\phi \in C^1(\ol\gO)$ and $z \in \pl\gO$ such that $u^*-\phi$ attains a local maximum {\rm (}resp., $u_*-\phi$ attains a local minimum{\rm )} at $z$,
\begin{gather*}
\min \{ \gl u^*(z) - \ep^{-1}b(z) \cdot D\phi(z) + G(z, D\phi(z)), u^*(z) - g(z) \} \leq 0 \\
(\text{resp., } \max \{ \gl u_*(z) - \ep^{-1}b(z) \cdot D\phi(z) + G(z, D\phi(z)), u_*(z) - g(z) \} \geq 0).
\end{gather*}
A function $u : \ol\gO \to \R$ is called a viscosity solution of \eqref{E: epHJ} and \eqref{E: epBC} if $u$ is both a viscosity sub- and supersolution of \eqref{E: epHJ} and \eqref{E: epBC}.
\end{defn}

Let $\cS_\ep$ (resp., $\cS_\ep^-$) denote the set of all viscosity solutions 
(resp., subsolutions) of (HJ$^\ep$) and (BC$^\ep$).

\begin{prop} \label{P: EB}
For each $\ep > 0$, there exists a viscosity solution $u^\ep$ of \eqref{E: epHJ} and \eqref{E: epBC}, that is, $\cS_\ep\not=\emptyset$. 
Furthermore, the set $\bigcup_{\ep>0} \cS_\ep$ 
is uniformly bounded on $\ol\gO$.   
\end{prop}

\def\x{\hat x}

\begin{proof}  Fix any $\ep>0$. We choose a constant $C>0$ so that 
\[
\max_{x\in\ol\gO}|G(x,0)|\leq \gl C \ \ \text{ and } \ \ \max_{x\in\pl\gO}|g(x)|\leq C, 
\]
and observe that $C$ and $-C$ are, respectively, a viscosity super- and subsolution of (HJ$^\ep)$ and (BC$^\ep$). Set 
\[
u^\ep(x)=\sup\{v(x)\mid v\in\cS_\ep^-,\ |v|\leq C \ \ON \ol\gO\} \ \ \text{ for } x\in\ol\gO,
\] 
and conclude by \cite{I87} that $u^\ep$ is a viscosity solution of (HJ$^\ep$) and (BC$^\ep$). 
Thus, $\cS_\ep\not=\emptyset$. 

Next, let $\ep>0$ and $v\in\cS_\ep$. Let 
$\x\in\ol\gO$ be a maximum point of $v^*$. 
If $x\in\gO$, then we have 
\[
\gl v^*(\x)+G(\x,0)\leq 0. 
\]
If, otherwise, $\x\in\pl\gO$, then, either, 
\[
\gl v^*(\x)+G(\x,0)\leq 0 \ \ \text{ or } v^*(\x)\leq g(\x).
\]
Hence, we get
\[
\sup_{\ol\gO}v=v^*(\x)\leq \max\{-\gl^{-1}\max_{x\in\ol\gO}G(x,0),\, \max_{\pl\gO}g\}. 
\]
Similarly, we obtain
\[
\min_{\ol\gO}v_*\geq \min\{-\gl \min_{x\in\ol\gO}G(x,0),\,\min_{\pl\gO}g\}. 
\]
Thus, we have 
\[
\sup_{\ol\gO}|v|\leq\max\{\gl^{-1}\max_{x\in\ol\gO}|G(x,0)|,\,\max_{\pl\gO}|g|\}, 
\]
which shows that $\bigcup_{\ep>0}\cS_\ep$ is uniformly bounded on $\ol\gO$. 
\end{proof}

The following example shows that, in general, viscosity solutions of \eqref{E: epHJ} and \eqref{E: epBC} do not satisfy the Dirichlet condition in the classical sense. 
Moreover, the uniqueness of the viscosity solutions of \eqref{E: epHJ} and \eqref{E: epBC} does not hold. 

\begin{exam}
Let $G$ and $g$ be the functions defined by $G(x, p) = |p|$ for $(x, p) \in \ol\gO \tim \R^2$ and $g(x) = 1$ for $x \in \pl\gO$, respectively. 
Then $u(x) \equiv 0$ is a viscosity solution of \eqref{E: epHJ} and \eqref{E: epBC}.
However it does not satisfy $u = 1$ on $\pl\gO$. 
If we set $v(x)=0$ for $x\in\gO$ and $v(x)=1$ for $x\in\pl\gO$, then the function $v$ is another viscosity solution of \eqref{E: epHJ} and \eqref{E: epBC}. 
\end{exam}

The following comparison theorem is a direct consequence of \cite[Theorem 2.1]{I89}. 

\begin{prop}\label{P: comp_ep} Assume \emph{(G1)--(G3)}. Let $u$ and $v$ be a viscosity sub- and supersolution of \eqref{E: epHJ} and \eqref{E: epBC}, respectively. 
If both $u$ or $v$ are continuous at the points of $\pl\gO$, then 
$u\leq v$ on $\ol\gO$. Also, if $u$ (resp., $v$) is continuous at the points of $\pl\gO$ and $u\leq g$ (resp., $v\geq g$) on $\pl\gO$, then $u\leq v$ on  $\ol\gO$. 
\end{prop}

We remark here that assumption (G5) does not ensure that 
$-\ep^{-1}b(x) \cdot p + G(x, p)$ is coercive when $\ep > 0$ is very small.
Assumption (G4) is assumed for technical reasons, and we do not know if such a 
convexity assumption on $G$ is needed or not to get the convergence result in our main theorem.

\begin{exam}
Consider the function $G$ defined by
\begin{equation*}
G(x, p) = \gth |p| - |p \cdot b(x)| - f(x) \ \ \  \For (x, p) \in \ol\gO \tim \R^2,
\end{equation*}
where $f \in C(\ol\gO)$ and $\gth > 0$ is chosen so that $\gth > \|DH\|_{\infty,\ol\gO}$.
It is easy to check that $G$ satisfies (G1)--(G5) and that, if $x\not=0$, $G(x, \cdot)$ is not convex.
\end{exam}

\section{Main result} \label{S: main result}
For $i\in\cI_0$, we set 
\begin{equation*}
c_i(h) = \{ x \in \ol\gO_i \mid H(x) = h \} \ \ \  \For h \in \bar J_i,
\end{equation*}
and define the function $\ol G_i\map (\bar J_i\setminus\{0\}) \tim \R\to \R$ by
\begin{equation*}
\ol G_i (h, q) = \cfr{1}{T_i (h)} \int_{\{ x \in \ol\gO_i \mid H(x) = h \}} \frac{G(x, qDH(x))}{|DH(x)|} \, dl,
\end{equation*}
where 
\begin{equation*}
T_i(h) = \int_{\{ x \in \ol\gO_i \mid H(x) = h \}} \frac{1}{|DH(x)|} \, dl
\end{equation*}
and $dl$ denotes the line element. 
We call the functions $\ol G_i$ the effective 
Hamiltonians.

Our main result, Theorem \ref{T: main} below, claims that the limit function of $u^\ep$, as $\ep \to 0+$, is characterized by the maximal viscosity solution $(u_0, u_1, \ldots, u_{N-1})$ to
\begin{equation} \tag{HJ}\label{HJ}
\left\{
\begin{minipage}{0.6\textwidth} 
\begin{align*} 
(\mathrm{HJ}_i)&& \hspace{.2\textwidth} \gl u_i + \ol G_i (h, u_i') &= 0 \ \ \  \In J_i,
\hspace{0.2\textwidth}\\
(\mathrm{BC}_i)&& u_i(h_i) &= \min_{\pl_i \gO} g,  \\
(\mathrm{NC})&& u_0(0) = u_1(0) &= \cdots = u_{N-1}(0), 
\end{align*}
\end{minipage}\right.
\end{equation} 

We recall the definition of viscosity solution of (HJ$_i$) and (BC$_i$). 

\begin{defn}
A function $u : J_i \to \R$ is called a viscosity subsolution {\rm (}resp., supersolution{\rm )} of (HJ$_i$) if $u$ is locally bounded in $J_i$ and, for any $\phi \in C^1(J_i)$ and $z \in J_i$ such that $u^*-\phi$ attains a local maximum {\rm (}resp., $u_*-\phi$ attains a local minimum{\rm )} at $z$,
\begin{gather*}
\gl u^*(z) + \ol G_i(z, \phi'(z)) \leq 0 \ \ \  (\text{resp., } \gl u_*(z) + \ol G_i(z, \phi'(z)) \geq 0).
\end{gather*}
A function $u : J_i \to \R$ is called a viscosity solution of (HJ$_i$) if $u$ is both a viscosity sub- and supersolution of (HJ$_i$).
\end{defn}

\begin{defn}
A function $u : \bar J_i \setminus \0 \to \R$ is called a viscosity subsolution {\rm (}resp., supersolution{\rm )} of (HJ$_i$) and (BC$_i$) if $u$ is locally 
bounded in $\bar J_i\setminus \{0\}$ and the following two conditions hold: 
{\rm (i)} $u$ is a viscosity subsolution {\rm (}resp., supersolution{\rm )} of (HJ$_i$), {\rm (ii)} for any $\phi \in C^1(\bar J_i \setminus \0)$ such that $u^*-\phi$ attains a local maximum {\rm (}resp., $u_*-\phi$ attains a local minimum{\rm )} at $h_i$,
\begin{gather*}
\min \{ \gl u^*(h_i) + \ol G_i(h_i, \phi'(h_i)), u^*(h_i) - \min_{\pl_i\gO}g \} \leq 0 \\
(\text{resp., } \max \{ \gl u_*(h_i) + \ol G_i(h_i, \phi'(h_i)), u_*(h_i) - g(h_i) \} \geq 0).
\end{gather*}
A function $u : \bar J_i \setminus \0 \to \R$ is called a viscosity solution of (HJ$_i$) and (BC$_i$) if $u$ is both a viscosity sub- and supersolution of (HJ$_i$) and (BC$_i$).
\end{defn}

We give the definition of (maximal) viscosity solutions of (HJ).

\begin{defn}
We say that $(u_0, u_1, \ldots, u_{N-1}) \in \prod_{i\in\cI_0}C(\bar J_i)$ is a viscosity solution (resp., subsolution) of {\rm (HJ)} if (NC) holds and, for each $i \in \cI_0$, $u_i$ is a viscosity solution (resp., subsolution) 
of (HJ$_i$) and (BC$_i$). Also, we say that $(u_0, u_1, \ldots, u_{N-1})$ is a maximal viscosity solution of {\rm (HJ)} provided it is a viscosity solution of {\rm (HJ)} and that, if $(v_0, v_1, \ldots, v_{N-1})$ is a viscosity solution of {\rm (HJ)}, then $u_i \geq v_i$ on $\bar J_i$ for all $i \in \cI_0$.
\end{defn}

We write $\cS$ (resp., $\cS^-$) for the set of all viscosity solutions (resp., subsolutions) $(u_0,\ldots,u_{N-1})\in \prod_{i\in\cI_0}C(\bar J_i)$ of (HJ). 

For any viscosity solution $(u_0,\ldots,u_{N-1})$ of \eqref{HJ}, we write 
\[
d(u_0,\ldots,u_{N-1}):=u_0(0)=\cdots=u_{N-1}(0). 
\]

It is clear that a maximal viscosity solution defined above is unique if it exists.

The main result in this paper is stated as follows.

\begin{thm} \label{T: main}
Assume that {\rm (G1)--(G5)} hold. \emph{(i)}\ There exists a maximal viscosity solution $(u_0,\ldots,u_{N-1})$ of {\rm(HJ)}. \ \emph{(ii)} \  
Define the function $u \in C(\ol\gO)$ by
\begin{equation*}
u(x) = u_i \circ H(x) \ \ \  \For x \in \ol\gO_i \text{ and } i \in \cI_0. 
\end{equation*}
Then the set $\cS_\ep$ converges to the function $u$ as $\ep\to 0+$ 
in the sense that for any compact subset $K$ of $\gO$,
\[
\lim_{\ep\to 0+}\sup \{\|v-u\|_{\infty,K}\mid v\in\cS_\ep\}=0. 
\]  
\end{thm}

The proof of this theorem is presented in Section \ref{S: proof main}.

\section{Effective problem (HJ$_i$) and (BC$_i$) in the edge $J_i$} 
\label{S: effective prob}

Hereafter, we always assume (G1)--(G5).
We study here some properties of the effective Hamiltonians $\ol G_i$ and the functions $T_i$ as well as viscosity subsolutions of the effective problem 
(HJ$_i$) and (BC$_i$) in the edge $J_i$.

\begin{lem} \label{L: PT}
Let $i \in \cI_0$. 
\begin{itemize}
\item[(i)] $T_i \in C^1 (\bar J_i \setminus \0)$.
\item[(ii)]
$T_i (h) = O(|h|^{-\frac{n}{m+2}})$ as $J_i \ni h \to 0$.
\end{itemize}
\end{lem}

We do not give here the proof of the lemma above, and refer for it to the proof of 
\cite[Lemmas 3.2 and 3.3]{Ku17}.

Since $n < m+2$, we see by  (ii) 
of Lemma \ref{L: PT} that
\begin{equation} \label{E: PT1}
T_i \in L^p(J_i) \ \ \   \If 1 \leq p < \frac{m+2}{n} \text{ and } \text{ for all } i \in \cI_0.
\end{equation}

\begin{lem} \label{L: PG}
Let $i \in \cI_0$.
\begin{itemize}
\item[(i)] $\ol G_i \in C(\bar J_i \setminus \0 \tim \R)$.

\item[(ii)] 
For any $h\in\bar J_i\setminus \{0\}$ and $q,q'\in\R$, 
\begin{equation*}
|\ol G_i(h, q) - \ol G_i(h, q')| \leq m_2(\max_{\ol\gO}|DH||q-q'|), 
\end{equation*}
where $m_2$ is the function from \emph{(G3)}. 
\item[(iii)] Let $\gamma$ be the positive number from \emph{(G4)}. For each $h \in J_i \cap (-\gam, \, \gam)$, the function $q \mapsto \ol G_i(h, q)$ is convex.

\item[(iv)] 
For every $(h,q)\in\bar J_i\setminus\{0\}\tim\R$, 
\begin{equation} \label{E: PG}
\ol G_i(h, q) \geq \frac{\nu L_i(h)}{T_i(h)}|q| - M,  
\end{equation} 
where $\nu,\,M$ are the constants from \emph{(G5)} and $L_i(h)$ denotes the length of $c_i(h)$, that is,
\begin{equation*}
L_i(h) = \int_{c_i(h)} \, dl.
\end{equation*}
\end{itemize}
\end{lem}

\begin{proof} We give an outline of the proof, and we leave it to the reader to check the details. 
Assertions (i), (ii), (iii), and (iv) follow from (G1) and (i) of Lemma \ref{L: PT}, (G3), (G2), and (G5),
respectively.  
\end{proof}


We note that $\ol G_i$ are locally coercive in $\bar J_i \setminus \0$ in the sense that,
for any closed interval $I$ of $\bar J_i \setminus \0$,
\begin{equation}\label{E: cor}
\lim_{r \to \infty} \inf \{ \ol G_i(h, q) \mid h \in I, \ |q| \geq r \} = \infty.
\end{equation}
This is an easy consequence of the fact that $L_i(h)\geq l_0$ for all $(h,i)\in \bar J_i\setminus\0\tim \cI_0$ and some constant 
$l_0>0$, Lemma \ref{L: PT}, and \eqref{E: PG}.

The next lemma is taken from \cite[Lemma 3.6]{Ku17}.

\begin{lem} \label{L: G0}
We have
\begin{equation*}
\lim_{J_i \ni h \to 0} \min_{q \in \R} \ol G_i(h, q) = \lim_{J_i \ni h \to 0} \ol G_i(h, 0) = G(0, 0) \ \ \  \text{ for all } i \in \cI_0.
\end{equation*}
\end{lem}


\begin{lem} \label{L: PS1}
Let $i \in \cI_0$ and $v \in \USC(J_i)$ be a viscosity subsolution of \emph{(HJ$_i$)}.
Then $u$ is uniformly continuous in $J_i$ and, hence,
it can be extended uniquely to $\bar J_i$ as a continuous function on $\bar J_i$.
Furthermore the extended function is also locally Lipschitz continuous in $\bar J_i \setminus \0$.
\end{lem}

\begin{lem} \label{L: PS2}
Let $i \in \cI_0$ and $\cF$ be a family of viscosity 
subsolutions of \emph{(HJ$_i$)}.
Assume that $\cF \cap C(\bar J_i)$ is uniformly bounded on $\bar J_i$.
Then $\cF \cap C(\bar J_i)$ is equi-continuous on $\bar J_i$.
\end{lem}

These two lemmas are easy consequences of \eqref{E: PT1} and \eqref{E: PG}. 
We refer to \cite[Lemmas 3.2--3.4]{Ku16} for the detail of the proof.

The local coercivity of $\ol G_i$ ensures that the classical inequalities hold at $h_i$ for any viscosity subsolutions of (HJ$_i$) and (BC$_i$).

\begin{lem} \label{L: PS4}
Let $i \in \cI_0$ and $v \in C(\bar J_i)$ be a viscosity subsolution of \emph{(HJ$_i$)} and \emph{(BC$_i$)}.
Then we have \ $v(h_i) \leq \min_{\pl_i\gO} g$.
\end{lem}

Thanks to Lemmas \ref{L: PG} and \ref{L: PS4}, 
the comparison principle is valid for (HJ$_i$) and (BC$_i$), 
as stated in the next lemma.  

\begin{lem} \label{L: limCP}
Let $i \in \cI_0$ and let $v \in C(\bar J_i)$ and $w \in \LSC(\bar J_i)$ be, respectively, a viscosity sub- and supersolution of \emph{(HJ$_i$)} and \emph{(BC$_i$)}.
Assume that $v(0) \leq w(0)$.
Then $v(h) \leq w(h)$ for all $h \in \bar J_i$.
\end{lem}

\section{Proof of the main theorem}\label{S: proof main}

We present the proof in two parts.

\begin{proof}[Proof of (i) of Theorem \ref{T: main}] In view of Lemmas \ref{L: PG} and 
\ref{L: G0}, we  may choose a constant $C>0$ so that 
\[
|\ol G_i(h,0)|\leq \gl C \ \ \text{ for all } h\in J_i,\, i\in\cI_0, 
\text{ and } \ \  |g(x)|\leq C \ \ \text{ for all } x\in\pl\gO.
\]
It is obvious that the $N$--tuple of the constant function $C$ and 
that of $-C$ are a viscosity super- and sub-solution of \eqref{HJ}, respectively.   
We may assume that $\gl C\geq M$, where $M$ is the constant from (G5).

Let $\cS_C^-$ denote the set of $(v_0,\ldots,v_{N-1})\in \prod_{i=0}^{N-1} C(\bar J_i)$ 
such that $v_i$ is a viscosity subsolution of (HJ$_i$) in $J_i$ for any $i\in\cI_0$,  
$v_0(0)=\ldots=v_{N-1}(0)$, and $|v_i|\leq C$ in $J_i$ for all $i\in\cI_0$. 

According to Lemma \ref{L: PS2}, the family $\cS_C^-$ is equi-continuous in the sense that for every $i\in\cI_0$, the family $\{v_i\in C(\bar J_i)\mid (v_0,\ldots,v_{N-1})
\in\cS_C^-\}$ is equi-continuous on $\bar J_i$. 
 Hence, setting 
\[
u_i(h)=\sup\{v_i(h)\mid (v_0,\ldots,v_{N-1})\in\cS_C^-\} \ \ \text{ for } h\in\bar J_i,\,i\in\cI_0,
\]
we see that $u:=(u_0,\ldots,u_{N-1})\in \prod_{i\in\cI_0}C(\bar J_i)$ and $u_0(0)=\ldots=u_{N-1}(0)$. 
Moreover, in view of the Perron method, 
we find that $u\in\cS$. 

To see the maximality of $(u_0,\ldots,u_{N-1})$, let $(v_0,\ldots,v_{N-1})$ be a 
viscosity solution of \eqref{HJ}. Note by (iii) of Lemma \ref{L: PG} that for any $i\in\cI_0$, we have,
in the viscosity sense,
\[
0\geq \gl v_i+\frac{\nu L_i(h)}{T_i(h)}|v_i')-M \geq \gl v_i-\gl C \ \ \IN J_i, 
\]
which implies that $v_i\leq C$ on $\bar J_i$. We set 
\[
w_i:=v_i\vee(-C)=\min\{v_i,\,-C\} \ \ \text{ on }\bar J_i,\,i\in\cI_0.
\]
It is easily seen that $(w_0,\ldots,w_{N-1})\in\cS_C^-$, and consequently, 
$v_i\leq w_i\leq u_i$ on $\bar J_i, i\in\cI_0$. Thus, $u$ is a maximal viscosity 
solution of \eqref{HJ}.  
\end{proof}

We need some preliminary observations before going into the proof of (ii) of Theorem \ref{T: main}.  

Since the set $\bigcup_{\ep>0}\cS_\ep$ is uniformly bounded on $\ol \gO$ by Proposition \ref{P: EB}, and hence, 
the half relaxed-limits $v^+$ and $v^-$ 
of $\cS_\ep$, as $\ep \to 0+$, 
\begin{equation} \label{E: 4.1}
\left\{\begin{aligned}
&v^+ (x) = \lim_{r \to 0+} \sup \{ u(y) \mid u\in\cS_\ep,\  
y \in B_r (x) \cap \ol \gO, \ \ep \in (0, r) \}, \\
&v^- (x) = \lim_{r \to 0+} \inf \{ u(y) \mid u\in\cS_\ep,\  y \in B_r (x) \cap \ol \gO, \ \ep \in (0, r) \}
\end{aligned}\right.
\end{equation}
are well-defined, bounded and, respectively, upper and lower semicontinuous on $\ol \gO$. 

For $i\in\cI_0$, we set
\begin{equation}\label{E: 4.2}
v^+_i(h)=\max_{c_i(h)}v^+ \ \ \text{ for } h\in \bar J_i\setminus \{h_i\} \ \ \text{ and } \ \ v_i^-(h)=\min_{c_i(h)}v^- \ \ \text{ for } h \in \bar J_i, 
\end{equation}
and \ $v_i^+(h_i)=\limsup_{J_i\ni h\to h_i}v_i^+(h)$. 

It is easily seen that $v_i^+\in\USC(\bar J_i)$ and $v_i^-\in\LSC(\bar J_i)$ for all $i\in\cI_0$. 

For the proof of (ii) of Theorem \ref{T: main}, the following three 
propositions are crucial.   
\begin{prop} \label{P: loop} For any $i\in\cI_0$ and $h\in J_i$, 
\begin{equation}\label{E: loop1}
v^+(x) =v_i^+(h)  \ \ \text{ and } \ \ v^-(x)=v_i^-(h) \ \ \text{ for all } x\in c_i(h).
\end{equation}
\end{prop}
\begin{thm} \label{T: Cl1}
For every $i \in \cI_0$, 
the functions $v_i^+$ and $v_i^-$ are, respectively, a viscosity sub- and 
supersolution of \emph{(HJ$_i$)} and \emph{(BC$_i$)}.
\end{thm}
\begin{thm} \label{T: Cl2} For any $i\in\cI_0$,   
\[
v^+_i(0) = v^-_i(0) = d(u_0,\ldots,u_{N-1}),
\]
where $(u_0,\ldots,u_{N-1})$ is the maximal viscosity solution of \eqref{HJ}. 
\end{thm}

Once these three propositions are in hand, the completion of the proof 
of (ii) of Theorem \ref{T: main} is easily done as follows.

\begin{proof}[Proof of (ii) of Theorem \ref{T: main}] By the definition of $v_i^\pm$, we have $v_i^-(h)\leq v_i^+(h)$ for all $h\in\bar J_i$ and $i\in\cI_0$. Hence, we deduce  by Theorem \ref{T: Cl2} and the semicontinuities of $v_i^\pm$ that the functions $v_i^+$ and $v_i^-$ are continuous at $h=0$. Now,  
since $v_i^+(0)=v_i^-(0)=u_i(0)$ by Theorem \ref{T: Cl2},  
Lemma \ref{L: limCP} ensures that $v_i^+\leq u_i\leq v_i^-$ on $\bar J_i$ for all $i\in\cI_0$,
which implies that $v_i^+=v_i^-=u_i$ on $\bar J_i$ for all $i\in\cI_0$. 
By the standard compactness argument together with Proposition \ref{P: loop}, we conclude that for any compact subset 
$K$ of $\gO$, we have 
\[
\lim_{\ep\to 0+}\sup\{\|u-w\|_{\infty,K}\mid w\in\cS_\ep\}=0. \qedhere
\]
\end{proof} 

We remark that the proof above shows that 
\[
\lim_{\ep\to 0+}\sup\{\|(w-u)_-\|_{\infty,\ol\gO}\mid w\in\cS_\ep \}=0,
\]
where $a_-$ denotes the negative part $\max\{0,-a\}$ for $a\in\R$.

It remains to prove Proposition \ref{P: loop}, Theorems \ref{T: Cl1}, and \ref{T: Cl2}, and we give the proof of Proposition \ref{P: loop}, Theorems \ref{T: Cl1}, 
and \ref{T: Cl2}, respectively, in this section, Sections \ref{S: visco}, and \ref{S: maximal}.

We consider the Hamiltonian flow associated with the Hamiltonian $H$: 
\begin{equation} \label{E: HS}
\dot X(t) = b(X(t)) \ \ \  \text{ and } \ \ \  X(0) = x \in \R^2,
\end{equation}
and write $X(t, x)$ for the solution of \eqref{E: HS}, which has a basic property: 
\begin{equation*}
H(X(t, x)) = H(x) \ \ \  \text{ for all } (t, x) \in \R \tim \R^2.
\end{equation*}
In particular, if $x\in c_i(h)$, with $h\in \bar J_i$ and $i\in\cI_0$, then 
\[
X(t,x)\in c_i(h) \ \ \ \text{ for } t\in\R.
\]
It follows from (H1) and (H2) that the curve $c_i(h)$ is $C^1$-diffeomorphic to circle $S^1$ 
for any $h\in\bar J_i\setminus\{0\}$ and $i\in\cI_0$.   
Moreover,  if $h\in \bar J_i\setminus\{0\}$ and $i\in\cI_0$, then 
$b(x)\not=0$ for all $x\in c_i(h)$ 
and $t\mapsto X(t,x)$ has a finite period for any $x\in c_i(h)$. 
Let $x\in c_i(h)$, with $i\in\cI_0$ and $h\in\bar J\i\setminus\{0\}$ 
and let $\tau_i>0$ denote the minimal period of $t\mapsto X(t,x)$. Observe that 
\[
\tau_i=\int_0^{\tau_i}dt=\int_0^{\tau_i} \frac{|\dot X(t,x)|\,dt}{|DH(X(t,x))|}
=\int_{c_i(h)} \frac{dl}{|DH(x)|}=T_i(h).  
\]   
Thus, if $h\in \bar J_i\setminus\{0\}$, with $i\in\cI_0$, and if $x\in c_i(h)$, then  
$T_i(h)$ equals to the minimal period of $t\mapsto X(t,x)$.

We note here that $\ol G_i$ can be rewritten as
\begin{equation*}
\ol G_i(h, q) = \frac{1}{T_i(h)}\int_0^{T_i(h)} G(X(t, x), qDH(X(t, x))) \, dt\ \ \
 \text{ if } \ h\not=0, 
\end{equation*} 
where $x \in c_i(h)$ is an arbitrary point. 
This representation says that $\ol G_i(h, q)$ is the average value of the periodic function 
$t\mapsto G(X(t,x),qDH(X(t,x)))$ for $x\in c_i(h)$ over the period $T_i(h)$.

\begin{proof}[Proof of Proposition \ref{P: loop}] We see immediately that  
$v^+$ and $v^-$ are a viscosity sub- and supersolution of 
\begin{equation} \label{E: transport}
-b \cdot Du =0 \ \ \  \In \gO,
\end{equation}
which, moreover, implies that $-v^-$ is a viscosity subsolution of \eqref{E: transport}. 
This observation ensures together with Proposition \ref{app1} in the appendix (or \cite[Theorem I.14]{CrLi83}) that for any $x\in\gO$, the functions $t\mapsto v^+(X(t,x))$ and $t\mapsto -v^-(X(t,x))$ are nondecreasing in $\R$. Hence, by the periodicity of $t\mapsto X(t,x)$, with $x\in c_i(h)$, $h\in J_i$, and $i\in\cI_0$, we infer that the functions $v^+$ and $v^-$ are constant on $c_i(h)$ for $h\in J_i,\,i\in\cI_0$.   
It is now clear that \eqref{E: loop1} holds. 
\end{proof}

\section{Viscosity properties of the functions $v_i^+$ and $v_i^-$} \label{S: visco}

We prove Theorem \ref{T: Cl1} in this section.  

Let $M$ be the positive constant from (G5), and in view of Proposition 
\ref{P: EB}, we define a positive number $C_M$ by 
\begin{equation} \label{E: C_M}
C_M=\max\{M,\, \sup\{\|u\|_{\infty,\ol\gO} \mid \, u\in \cS_\ep, \,\ep>0\}\}. 
\end{equation}

The next lemma is a quantitative version of Proposition \ref{P: loop}. 

\begin{lem} \label{L: loop} There is a constant $C>0$ such that 
for any $\ep>0$, $u\in\cS_\ep$, $i\in\cI_0$, and $h\in J_i$, 
\[
|u^*(x)-u^*(y)|\leq \ep C T_i(h) \ \ \text{ for all } x,y\in c_i(h).
\]
\end{lem} 

\begin{proof} Fix any $\ep>0$, $u\in\cS_\ep$, $i\in\cI_0$, $h\in J_i$, and $x,y\in c_i(h)$.  
The trajectory $t\mapsto X(t,x)$ stays in $c_i(h)$ and for some $\tau\in(0, 2T_i(h)]$, 
it meets $y$ at $t=\tau$, that is, $X(\tau, x)=y$.  Since $u^*$ is a  
viscosity subsolution of \ $\gl u^*-\ep^{-1}\,b\cdot Du^*-M=0$ \ in $\gO$\ 
by (G5), and $Y(t)=X(\ep^{-1} t,x)$ satisfies
\[
\dot Y(t)=\frac{1}{\ep}\, b(Y(t)) \ \ \text{ for all } t\in\R,
\]
we deduce by Proposition \ref{app1} that 
\[\bald
u^*(x)&\,\leq e^{-\ep \gl \tau}u^*(Y(\ep \tau))+\int_0^{\ep\tau} e^{-\gl t} M\,dt 
\\&\,\leq u^*(y) +C_M(1-e^{-\ep \gl\tau}) +M\ep \tau\leq u^*(y)+ 
2\ep C_M(\gl+1)T_i(h).  
\eald\]
Thus, by the symmetry in $x$ and $y$, we obtain 
\[
|u^*(x)-u^*(y)|\leq 2\ep C_M(\gl+1) T_i(h). \qedhere
\]
\end{proof} 

\begin{lem} \label{L: visco1} For any $\ep>0$, $u\in\cS_\ep$, and $y\in\pl\gO$, 
we have \ $u^*(y)\leq g(y)$. 
\end{lem}

\begin{proof} Fix any $\ep>0$, $u\in\cS_\ep$ and $y\in\pl\gO$. Choose $i\in\cI_0$ so that $y\in\pl_i\gO=c_i(h_i)$. 
Note by (G5) that $u^*$ is a viscosity subsolution of 
\begin{equation}\label{E: convex}
\gl u^*-\ep^{-1}\,b\cdot Du^*+\nu |Du^*|-M=0 \ \ \IN \gO \ \ \text{ and } \ \ u^*=g \ \ \ON\pl\gO,
\end{equation}

For $\ga>0$ and $\gb>0$, we set 
\[
\phi_{\ga,\gb}(x)=\ga|x-y|^2+\gb|H(x)-h_i| \ \ \text{ for } x\in \ol \gO_i.
\]
Let $x_{\ga,\gb}\in\ol\gO_i$ be a maximum point of the function $u^*-\phi_{\ga,\gb}$ on $\ol\gO_i$. It is easily seen that
\[
\lim_{\ga,\gb\to \infty}x_{\ga,\gb}=y \ \ \text{ and } \ \ \lim_{\ga,\gb\to\infty}u^*(x_{\ga,\gb})=u^*(y).
\]

Fix $r>0$ so that $\dist(B_r(y), c_i(0))>0$ and hence, 
$\inf_{B_r(y)\cap \gO_i}|DH|>0$. 
We fix $\ga_0>0$ so that if $\ga,\gb\in(\ga_0,\infty)$, then $x_{\ga,\gb}\in B_r(y)$.  

For $x\in B_r(y)\cap \ol\gO_i$, we compute that  
\[\bald
\gl u^*(x)&-\ep^{-1}\,b(x)\cdot D\phi_{\ga,\gb}(x)+\nu|D\phi_{\ga,\gb}(x)|-M
\\&\,
=\gl u^*(x)-\ep^{-1}\,b(x)\cdot\Big(2\ga(x-y)+\gb\,\frac{H(x)-h_i}{|H(x)-h_i|}\,DH(x)\Big)
\\&+\nu\Big|2\ga(x-y)+\gb \,\frac{H(x)-h_i}{|H(x)-h_i|}\,DH(x)\Big|-M
\\&\,
=\gl u^*(x)-2\ga\,\ep^{-1}\,b(x)\cdot (x-y)
+\nu\Big|2\ga(x-y)+\gb \,\frac{H(x)-h_i}{|H(x)-h_i|}\,DH(x)\Big|-M
\\&\,=\gl u^*(x)-2\ga\,\ep^{-1}\, r\|b\|_{\infty,\ol\gO}
+\gb\nu \inf_{B_r(y)\cap \gO_i}|DH| -2\ga \nu r -M.
\eald \]
Hence, for any $\ga>\ga_0$, we may choose $\gb=\gb(\ga)>\ga$ so that 
\[
\gl u^*(x_{\ga,\gb})-\ep^{-1}\,b(x_{\ga,\gb})\cdot D\phi_{\ga,\gb}(x_{\ga,\gb})
+\nu|D\phi_{\ga,\gb}(x_{\ga,\gb})|-M>0. 
\] 
Now, we deduce from \eqref{E: convex} that for any $\ga>\ga_0$,  
\[
x_{\ga,\gb(\ga)}\in \pl_i\gO \ \ \ \text{ and }  \ \ \ u^*(x_{\ga,\gb(\ga)})
\leq g(x_{\ga,\gb(\ga)}).
\] 
Sending $\ga\to\infty$, we conclude that $u^*(y)\leq g(y)$. 
\end{proof}

\begin{lem} \label{L: Ch2} For every $i\in\cI_0$,
\begin{equation} \label{E: Ch2}
v_i^+(h_i) \leq \min_{\pl_i\gO} g. 
\end{equation}
\end{lem}

\begin{proof} 
We give the proof of \eqref{E: Ch2} only for $i=0$ since we can prove the others similarly.

Fix any $h\in J_0$ and $y\in c_0(h)$. By Proposition \ref{P: loop}, we have 
$v_0^+(h)=v^+(y)$.  We select sequences of $\ep_k>0$, $y_k\in\gO_0$, and 
$u_k\in \cS_{\ep_k}$, with 
$k\in \N$, so that 
\[\lim_{k\to \infty} (\ep_k, y_k, u_k^*(y_k))=(0,y,v^+(y)).
\]
We set $\gam_k=H(y_k)\in J_0$ for $k\in\N$. 

Let $z \in \pl_0\gO$ be a minimum point of $g$ over $\pl_0\gO$, and fix $k\in\N$.  
Consider the initial value problem
\begin{equation} \label{E: Ch21}
\dot Z(t) =\frac{1}{\ep_k} \, b(Z(t)) + \nu \, F(Z(t)) \ \ \  \text{ and } \ \ \  
Z(0) = z,
\end{equation}
where $F(x):=DH(x)/|DH(x)|$.   
This problem has a unique solution $Z(t)$ 
as long as $Z(t)$ is away from any of critical points of $H$.  
Let $I$ be the maximal existence interval of the solution $Z(t)$. 

Note that
\begin{equation}\label{E: Ch22}
\frac{d}{dt}H(Z(t))=DH(Z(t))\cdot \dot Z(t)
=\nu\, |DH(Z(t))|>0 \ \ \text{ for all } t\in I,
\end{equation}
and hence the function $t\mapsto H(Z(t))$ is increasing in $I$. 
Since the origin is the only critical point of $H$ in $\ol\gO_0$ and $H(0)=0$, we deduce that there is $\gs\in I$, with $\gs<0$, such that 
$0<H(Z(\gs))=\gam_k$. Moreover, we have \ $Z(t)\in\gO_0$ for all 
$t\in(\gs,\,0)$. 

We may assume, by reselecting the sequence $\{(\ep_k,y_k,u_k)\}_{k\in\N}$ if necessary, that $\gam_k>h_0/2$ for all $k\in\N$. 
There exists a constant $\gd>0$ such that $|DH(x)|>\gd$ for all 
$x\in \gO_0$ satisfying $H(x)>h_0/2$.   
It follows from \eqref{E: Ch22} that 
\begin{equation} \label{E: Ch23}
h_0-\gam_k\geq \nu\gd |\gs|.  
\end{equation}

Note that $u_k^*$ is a viscosity subsolution of 
\[
\gl u_k^*-\Big(\frac{b}{\ep_k}+\nu F\Big)\cdot Du_k^*-M=0 \ \ \IN \gO. 
\]
Set $z_k=Z(\gs)$. By Proposition \ref{app1}, we obtain
\[
e^{-\gl \gs}u_k^*(z_k)\leq e^{-\gl t}u_k^*(Z(t))+\int_{\gs}^t e^{-\gl s} M \,ds \ \ \text{ for all } t\in (\gs,\, 0),  
\]
which implies, in the limit as $t\to 0-$, that 
\[
u_k^*(z_k)\leq e^{\gl \gs}\left(u_k^*(z)+M|\gs|\right)
\leq u_k^*(z)+C_M (1-e^{-\gl|\gs|})+M|\gs|.
\]
Combining this with Lemmas \ref{L: loop} and \ref{L: visco1}, we get 
\[
u_k^*(y_k)\leq \ep_kCT_0(\gam_k) + u_k^*(z_k) 
\leq \ep_k CT_0(\gam_k) +g(z)+(\gl C_M+M)|\gs|
\]
for some constant $C>0$, and moreover, by \eqref{E: Ch23},
\[
u_k^*(y_k)\leq \ep_k CT_0(\gam_k) +g(z)+(\gl C_M+M)\gd^{-1}(h_0-\gam_k),
\]
Sending $k\to \infty$ yields 
\[
v^+_0(h)=v^+(y)\leq g(z) + (\gl C_M+M)\gd^{-1}(h_0-h). 
\]
Consequently,
\[
v_0^+(h_0)=\limsup_{J_0\ni h\to h_0}v^+_0(h)\leq g(z)=\min_{\pl_i\gO}g. \qedhere
\]
\end{proof}

The next lemma is proved in the proof of \cite[Theorem 3.6]{Ku16}. 
For any $\ga<\gb$ and $i\in\cI_0$, we write $\gO_i(\ga,\gb)$ 
and $\ol\gO_i(\ga,\gb)$ for the sets 
$\{x\in \gO_i\mid \ga<H(x)<\gb\}$ and $\{x\in \ol\gO_i\mid \ga\leq H(x)\leq \gb\}$, respectively.

\begin{lem} \label{L: exist1} Let  $i\in\cI_0$, $h\in \bar J_i\setminus\{0\}$, and 
$q\in\R$.  For any $\gd>0$, there exist an interval $[\ga,\,\gb]\subset 
\bar J_i\setminus \{0\}$ and $\psi\in C^1(\ol\gO_i(\ga,\gb))$ such that 
$[\ga,\,\gb]$ is a neighborhood of $h$, relative to $\bar J_i\setminus\{0\}$, and 
\[   
\left|-b(x)\cdot D\psi(x)+G(x,qDH(x))-\ol G_i(H(x),q)\right|\leq \gd \ \ \text{ for all } 
x\in \ol\gO_i(\ga,\gb).
\]
\end{lem}

\begin{proof}[Proof of Theorem \ref{T: Cl1}] We follow the proof of \cite[Theorem 3.6]{Ku16}, which is based on the perturbed test function method due to \cite{E89}. We show that $v_0^-$ is a viscosity supersolution of (HJ$_0$) and (BC$_0$). 
A parallel argument shows that $v_i^-$, with $i\in\cI_1$, is a viscosity supersolution 
of (HJ$_i$) and (BC$_i$), the detail of which we omit presenting here.    

Let $\phi\in C^1(\bar J_0\setminus\{0\})$ and assume that $v_0^--\phi$ has a strict minimum at $\hat h$. Since the treatment for the case when $\hat h<h_0$ is 
similar to and easier than the case when $\hat h=h_0$, we, henceforth, consider only the case when $\hat h=h_0$. 

We need to show that either 
\[
\gl v_0^-(\hat h)+\ol G_0(\hat h,\phi'(\hat h))\geq 0 \ \ \text{ or } \ \ 
v_0^-(\hat h)\geq \min_{\pl_0\gO}g. 
\]
For this, we suppose that 
\begin{equation}\label{E: Ch24}
v_0^-(\hat h)<\min_{\pl_0\gO} g, 
\end{equation}
and prove that 
\begin{equation}\label{E: Ch25}
\gl v_0^-(\hat h)+\ol G_0(\hat h,\phi'(\hat h))\geq 0. 
\end{equation}

Fix any $\gd>0$ and set $q=\phi'(\hat h)$. 
By Lemma \ref{L: exist1}, there exist $\ga\in(0,\,\hat h)$ and 
$\psi\in C^1(\ol\gO_0(\ga,\hat h))$ such that
\[
-b(x)\cdot D\psi(x)+G(x,qDH(x))-\ol G_0(H(x),q)<\gd \ \ 
\text{ for all } x\in\ol \gO_0(\ga,\hat h).
 \] 

Recalling that $v_0^-(\hat h)=\min_{x\in c_0(\hat h)} v^-(x)$, we select $\hat x\in c_0(\hat h)$ so that $v_0^-(\hat h)=v^-(\hat x)$. 
We next select sequences of $\ep_k>0$, $x_k\in \ol\gO_0(\ga,\hat h)$, 
and $u_k\in\cS_{\ep_k}$, with $k\in\N$, so that 
\[
\lim_{k\to\infty}(\ep_k, x_k, (u_k)_*(x_k))=(0, \hat x, v^-(\hat x)). 
\]

For $k\in\N$, we consider the function 
\[
\Phi_k(x):=(u_k)_*(x)-\phi(H(x)) -\ep_k\psi(x)  \ \ \ON \ol\gO_i(\ga,\hat h).
\]
This function is lower semicontinuous and has a minimum at some point $y_k$.
We may assume, by relabeling the sequences if needed, that $\{y_k\}_{k\in\N}$ 
converges to some point $y_0\in\ol\gO_0(\ga,\hat h)$.

Noting that $\Phi_k(x_k)\geq \Phi_k(y_k)$ for all $k\in\N$,  
\[
\lim_{k\to \infty}\Phi_k(x_k)= v^-(\hat x)-\phi(H(\hat x)) 
=(v_0^--\phi)(\hat h),
\]
and
\[
\liminf_{k\to\infty}\Phi_k(y_k)\geq v^-(y_0)-\phi(H(y_0))
\geq (v_0^--\phi)(H(y_0)),
\]
we deduce that
\[
\lim_{k\to \infty}((u_k)_*(y_k)-\phi(H(y_k)))=(v_0^--\phi)(\hat h),\ \ 
\lim_{k\to \infty}(u_k)_*(y_k)=v_0^-(\hat h) \ \ \text{ and } \ \  y_0\in c_0(\hat h)). 
\]
Thanks to \eqref{E: Ch24}, we may assume without loss of generality that 
\[
(u_k)_*(y_k)<\min_{\pl_0\gO} g,
\]
and, by the viscosity property of $(u_k)_*$ and by choice of $\psi$, we obtain
\[\bald
0&\,\leq \gl (u_k)_*(y_k) -\frac{1}{\ep_k}\,b(y_k)\cdot (\phi'(H(y_k))DH(y_k)
+\ep_k D\psi(y_k)) 
\\&\,\quad +G(y_k, \phi'(H(y_k))DH(y_k)+\ep_kD\psi(y_k))
\\&\,=\gl (u_k)_*(y_k) -b(y_k)\cdot D\psi(y_k)) 
+G(y_k, \phi'(H(y_k))DH(y_k)+\ep_kD\psi(y_k))
\\&\,\leq \gl (u_k)_*(y_k) +\gd -G(y_k,qDH(y_k))+\ol G_0(H(y_k),q)
\\&\,\quad +G(y_k, \phi'(H(y_k))DH(y_k)+\ep_kD\psi(y_k)).
\eald\]
Hence, in the limit as $k\to\infty$, we obtain
\[
-\gd \leq \gl v_0^-(\hat h)+\ol G_0(\hat h,q),
\]
which proves \eqref{E: Ch25}.

According to Lemma \ref{L: Ch2}, we have\ $v_i^+(h_i)\leq \min_{\pl_i\gO}g$ for all 
$i\in\cI_0$. Hence, it remains to show that $v_i^+$, with $i\in\cI_0$, is a 
viscosity subsolution of (HJ$_i$).  The argument presented above is easily adapted to show this, the detail of which we leave it to the reader to check.   
\end{proof}

\section{The maximality of the viscosity solution $(v_0^+,\ldots,v_{N-1}^+)$} \label{S: maximal}

Due to Theorem \ref{T: Cl1} and Lemma \ref{L: PS1}, the functions $v_i^+$, with $i\in\cI_0$, are continuous on $\bar J_i\setminus \{0\}$ and have the limit $\,\lim_{J_i\ni h\to 0}v_i^+(h)\in\R$. We set 
\[
d(v_i^+)=\lim_{J_i\ni h\to 0}v_i^+(h) \ \ \ \text{ for } i\in\cI_0. 
\]

\begin{lem} \label{L: max1} For any $i\in\cI_1$, 
\[
\inf_{x\in c_i(0)}v^+(x)\geq 
\max\{d(v_i^+),\,d(v_0^+)\}. 
\]
\end{lem}

\begin{proof} Fix any $i\in\cI_1$ and $x\in c_i(0)$.  Fix any $\gd>0$, and choose $r>0$ so that 
\[
v^+(x)+\gd>\sup\{u(y)\mid u\in\cS_\ep,\ y\in \ol\gO\cap B_r(x),\ 0<\ep<r \}.  
\] 

We choose $h_{i,\gd}\in J_i$ and $h_{0,\gd}\in J_0$ so that 
\[
B_r(x)\cap c_i(h_{i,\gd})\not=\emptyset \ \ \text{ and } \ \ 
B_r(x)\cap c_0(h_{0,\gd})\not=\emptyset. 
\]
and that 
\[
v_i^+(h_{i,\gd}) +\gd> d(v_i^+) \ \ \text{ and } \ \  
v_0^+(h_{0,\gd})+\gd>d(v_0^+). 
\]

By Proposition \ref{P: loop}, we have 
\[
v_i^+(h_{i,\gd})=v^+(x) \ \ \text{ for all } x\in c_i(h_{i,\gd}).
\]
Hence, we may choose $x_\gd \in B_r(x)$ and 
$u_\gd\in\cS_{\ep_{\gd}}$, with  
$0<\ep_\gd<r$, such that 
\[
u_\gd(x_\gd)+\gd>v_i^+(h_{i,\gd}). 
\] 
Combining these observations, we obtain
\[
v^+(x)+3\gd>u_\gd(x_\gd)+2\gd>v_i^+(h_{i,\gd})+\gd>d(v_i^+),
\]
from which we conclude that 
\[
\inf_{x\in c_i(0)}v^+(x)\geq d(v_i^+).
\]

An argument similar to the above yields
\[
\inf_{x\in c_i(0)}v^+(x)\geq d(v_0^+),
\]
which completes the proof. 
\end{proof}

\begin{lem} \label{L: max3} We have
\begin{equation}\label{E: max21}
\max_{x\in c_0(0)}v^+(x)\leq d(v_0^+). 
\end{equation}
\end{lem}

We need the following two lemmas for the proof of Lemma \ref{L: max3}.

\begin{lem} \label{L: max9}There exists a constant $A_0>0$ such that 
\[
|DH(x)|\geq A_0 |H(x)|^\ga \ \ \ \text{ for all } x\in\gO,
\]
where $\ga:=n/(m+2)\in (0,\,1)$ and the constants $n,\, m$ are from \emph{(H3)}. 
\end{lem}

\begin{proof}  Let $m,\,n,\,A_1,\,A_2$, and $V$ be the constants 
and neighborhood of the origin from (H3), respectively.  
We may assume that $V=B_R$ for some $R>0$. 
Since $H(0)=0$ and $DH(0)=0$, we deduce by (H3) that 
\[
|H(x)|\leq C|x|^{m+2} \ \ \ \text{ for all } x\in B_R
\]
and some constant $C>0$, and consequently,
\[
|DH(x)|\geq A_2|x|^n\geq A_2\left(\frac{|H(x)|}{C}\right)^{\frac{n}{m+2}}
=\frac{A_2}{C^\ga}\,|H(x)|^\ga \ \ \text{ for all } x\in B_R. 
\] 
Noting that
\[
\min_{x\in \ol\gO\setminus B_R}\frac {|DH(x)|}{|H(x)|^\ga}>0,
\]
we conclude that for some constant $A_0>0$, 
\[
|DH(x)|\geq A_0|H(x)|^\ga \ \ \text{ for all } x\in\gO. \qedhere
\]
\end{proof}

We define the function $m_H\map [0,\,\infty)\to \R$ by
\[
m_H(r)=\max_{x\in\ol\gO_0\cap \ol B_r}H(x).
\]
Note that $m_H(r)>0$ for all $r>0$. 

\begin{lem} \label{L: mH}  Let $R>0$ be a constant such that $B_R\subset \gO$.
Then there exist constants $\rho>1$ and $A_3>0$ such that
\begin{equation}\label{E: mH}
m_H(r)\geq A_3r^\rho \ \ \ \text{ for all } \ r\in(0,\,R).
\end{equation}
\end{lem}

\begin{proof} Fix any $r,s\in(0,\,R)$ so that $r>s$, and choose a point $x_s\in \ol \gO_0\cap\ol B_s$ so that 
\[
m_H(s)=H(x_s). 
\]
Since $DH(x)\not=0$ for $x\in \gO\setminus \{0\}$, it follows that $H$ does not take a 
local maximum at any point in $\gO\setminus \{0\}$ and hence, $m_H(r)>m_H(s)$.  
More generally, the function $m_H$ is increasing in $(0\,R)$.  

Solve the initial value problem
\[
\dot Y(t)=F(Y(t)) \ \ \text{ and }  \ \ Y(0)=x_s,
\]
where $F$ is the function given by $F(x):=DH(x)/|DH(x)|$.
We note that 
\begin{equation}\label{E: mH1}
\frac{d}{dt}H(Y(t))=|DH(Y(t))|>0 \ \ \text{ for all } t\geq 0
\end{equation}
as far as $Y(t)$ exists, and we infer that $H(Y(t))\geq m_H(s)$ for all $t\geq 0$, and that there exists $\tau>0$ such that $H(Y(\tau))=m_H(r)$. 
From these, we deduce, together with the strict monotonicity of $m_H$, that 
$|Y(t)|\geq s$ for all $t\geq 0$, and $|Y(\tau)|=r$.  

Noting by Lemma \ref{L: max9} that $|DH(x)|\geq A_0|H(x)|^\ga$ for all $x\in\gO$ 
and some constants $A_0>0$ and $\ga\in(0,\,1)$, we compute by \eqref{E: mH1} that
\[\bald
m_H(r)^{1-\ga}-m_H(s)^{1-\ga}&\,=H(Y(\tau))^{1-\ga}-H(Y(0))^{1-\ga}
\\&\,=(1-\ga)\int_0^\tau H(Y(t))^{-\ga}\frac{d}{dt}H(Y(t))\,dt
\geq (1-\ga)A_0\tau.
\eald
\]
and that, since $|\dot Y(t)|=1$,
\[
r-s\leq |Y(\tau)|-|Y(0)|\leq |Y(\tau)-Y(0)|\leq \int_0^\tau |\dot Y(t)|dt= \tau.
\]
Hence, we obtain
\[
m_H(r)^{1-\ga}-m_H(s)^{1-\ga}\geq (1-\ga)A_0(r-s).
\]
Sending $s\to 0+$ yields
\[
m_H(r)\geq \left((1-\ga)A_0 r\right)^{\frac 1 {1-\ga}}=A_3r^\rho,
\]
where $\rho:=1/(1-\ga)$ and $A_3:=((1-\ga)A_0)^\rho$, 
which completes the proof. 
\end{proof}

\begin{proof}[Proof of Lemma  \ref{L: max3}]
Fix any $\eta > 0$ and choose $\gd_0\in J_0=(0,\, h_0)$ so that 
\[
d(v^+_0)+\eta>v_0^+(h) \ \ \text{ for all } h \in (0,\,\gd_0). 
\] 
We may assume that $\gd_0<\eta$ and $\gd_0<\bar h$. By the definition of $v_0^+$, we infer that for each $h\in (0,\,\gd_0)$, there is $\ep(h)>0$ such that if $h\in(0,\,\gd_0)$, then 
\begin{equation}\label{E: max22}
d(v_0^+)+\eta>\sup\{u^*(x)\mid u\in \cS_\ep,\, 0<\ep<\ep(h), \, x\in c_0(h)\}. 
\end{equation}

Fix any $\gd\in (0,\,\gd_0)$.  We choose a continuous nondecreasing function 
$f\map (-\infty,\,\gd) \to \R$ so that 
\[
f(r)=1 \ \ \text{ for } r<\gd/2 \ \ \text{ and } \ \ \lim_{r\to\gd-}f(r)=\infty.
\] 
Define $g\map (-\infty,\,\gd)\to \R$ by
\[
g(r)=\int_0^r f(t)\,dt. 
\]
Observe that 
\[
g(r)=r \ \ \text{ for } r\leq \gd/2 \ \ \text{ and } \ \ |r|\leq |g(r)|\leq g'(r)|r| \ \ \text{ for } 
r\in (-\infty,\,\gd).
\]

According to Lemma \ref{L: max9}, 
there are constants $\ga\in (0,\,1)$ and $A_0>0$ such that  
\begin{equation}\label{E: max23}
|DH(x)|\geq A_0 |H(x)|^\ga \ \ \text{ for all } x\in \ol\gO. 
\end{equation} 
Let $\gb\in(0,\,1)$ be a constant to be fixed later. 
We define the function $w\in  C(\gO(\gd)\cup c_0(\gd))$ by
\begin{equation*}
w(x) = g(-H(x)) |H(x)|^{\gb -1}+\gd^\gb 
+ d(v_0^+) + \eta.
\end{equation*}
Observe that
\[\bald
& w\in C^1(\gO(\gd)\setminus c_0(0)),
\\&\pl\gO_0(\gd)=c_0(\gd) \cup \bigcup_{j\in\cI_1}c_j(-\gd), 
\\&w(x)=d(v_0^+)+\eta \ \ \text{ for all } x\in c_0(\gd),
\\&\lim_{\gO(\gd)\ni y\to x} w(y)=\infty \ \ \text{ uniformly for } x\in \bigcup_{j\in \cI_1}c_j(-\gd).
\eald
\]
 
Compute that for $x\in \gO(\gd)\setminus c_0(0)$,
\[
\bald
Dw(x) &\,= \big[-g'(-H)|H|^{\gb-1}+(\gb-1)g(-H)|H|^{\gb-3}H\big]DH,
\\&\,= \big[g'(-H)(-H)+(\gb-1)g(-H)\big]|H|^{\gb-3}HDH
\eald
\]
and moreover,
\[
\bald
|Dw(x)| &\,\geq \left(g'(-H)|H|-(1-\gb)|g(-H)|\right)|H|^{-\gb-2}|DH|
\\&\,= \gb g(-H)|H|^{\gb-2}|DH|.
\eald
\]
Combining this with \eqref{E: max23} yields 
\[
|Dw(x)|\geq \gb A_0 |g(-H(x))| |H(x)|^{\ga+\gb-2}
\geq \gb A_0|H(x)|^{\ga+\gb-1}.
\]
Moreover, using (G5), we compute 
\begin{equation}\label{E: max24}
\gl w(x) - \ep^{-1} b(x) \cdot Dw(x) + G(x, Dw(x)) 
\geq \gl d(v_0^+) + \nu \gb A_0|H(x)|^{\ga +\gb-1} - M
\end{equation}
for all $x \in \gO(\gd)\setminus c_0(0)$. 

We assume in what follows that $\gb>0$ is sufficiently small so that 
\begin{equation}\label{E: max25}
\ga+\gb-1<0.
\end{equation}
In view of \eqref{E: max24}, by choosing $\gd\in(0,\,\gd_0)$ sufficiently small,  
we may assume that 
\begin{equation}\label{E: max26}
\gl w-\ep^{-1} b(x)\cdot Dw(x)+G(x,Dw(x))>0 \ \ \text{ for all } x\in\gO(\gd)\setminus c_0(0). 
\end{equation}

By Lemma \ref{L: mH}, we have
\begin{equation}\label{E: max27}
m_H(r)\geq A_3r^\rho \ \ \text{ for all } r\in(0,\,R),
\end{equation}
where $\rho>1$, $A_3>0$, and $R>0$ are constants. In addition to \eqref{E: max25}, we assume hereafter that $\gb<1/\rho$. That is, we fix $\gb>0$ so that 
\[
\gb<\min\{\rho^{-1},\,1-\ga\}. 
\]

We claim that
\begin{equation}\label{E: max28}
D^-w(x)=\emptyset \ \ \ \text{ for all } x\in c_0(0),
\end{equation}
where $D^-w(x)$ denotes the subdifferential of $w$ at $x$. 

To see this, we fix any $x\in c_0(0)$. By contradiction, we suppose that 
$D^-w(x)\not=\emptyset$. Let $\phi\in C^1(\gO(\gd))$ be a function such that 
$w-\phi$ attains a minimum at $x$. If $x\not=0$, then 
\[
x+tDH(x)\in \gO_0 \cap\gO(\gd) \ \ \text{ for all } t\in(0,\,t_0)
\]
and some $t_0>0$, and consequently, we have for $t\in (0,\,t_0)$, 
\[
(w-\phi)(x)\leq (w-\phi)(x+t DH(x)). 
\]
For sufficiently small $t>0$, this reads
\[t^{-\gb}\left(\phi(x)-\phi(x+tDH(x))\right)\geq 
t^{-\gb}\left(w(x)-w(x+tDH(x))\right)=t^{-\gb}H(x+tDH(x))^\gb, 
\]
which yields, in the limit as $t\to 0+$,
\[
0\geq |DH(x)|^{2\gb}. 
\] 
This is a contradiction. Otherwise, we have $x=0$ and, for any $y\in \gO(\gd)$,
\[
w(x)-w(y)\leq \phi(x)-\phi(y). 
\]
Moreover, for any $y\in \gO(\gd)\cap \gO_0$, we have
\[
H(y)^\gb\leq \phi(x)-\phi(y),
\]
and for any $r\in(0,\,\gd \wedge R)$, 
\[
m_H(r)^\gb\leq \max_{y\in \ol B_r\cap\ol\gO_0}(\phi(x)-\phi(y)).
\]
Since \ $
m_H(r)^\gb\geq A_3^\gb r^{\gb \rho}$  \ by \eqref{E: max27}
and $\gb \rho<1$, we obtain from the above 
\[
A_3^{\gb}\leq \lim_{r\to 0+} r^{-\gb \rho}\max_{y\in \ol B_r\cap\ol\gO_0}(\phi(x)-\phi(y))=0,
\]
which is a contradiction. Thus, we conclude that \eqref{E: max28} is valid,  
and moreover from \eqref{E: max26} and \eqref{E: max28} that $w$ 
is a viscosity supersolution of 
\[
\gl w-\ep^{-1}b\cdot Dw+G(x,Dw)\geq 0 \ \ \IN \gO(\gd). 
\] 

Recalling \eqref{E: max22}, we deduce 
by the comparison theorem that for any $\ep\in (0,\,\ep(\gd))$ 
and $u\in \cS_\ep$, we have
\[
u^*(x)\leq w(x) \ \ \text{ for all } x\in \gO(\gd),
\]
which yields 
\[
v^+(x)\leq w(x)=\gd^\gb+d(v_0^+)+\eta \ \ \text{ for all } x\in c_0(0). 
\]
This ensures that $v^+(x)\leq d(v_0^+)$ for all $x\in c_0(0)$.
\end{proof} 

\begin{lem}\label{L: max4} For every $i\in\cI_1$,
\[
d(v_0^+)\leq d(v_i^+). 
\]
\end{lem}

\begin{proof}  Fix $i\in\cI_1$, $z\in c_i(0)\setminus\{0\}$, and $\gd>0$ so that $\gd<h_0\wedge |h_i|$. We choose sequences of $\ep_k>0$, $u_k\in\cS_{\ep_k}$, and $x_k\in\gO_0$
such that as $k\to\infty$, 
\[
(\ep_k,H(x_k), u_k^*(x_k)) \to (0,\gd,v_0^+(\gd)). 
\] 
We set $\gam_k=H(x_k)$ and, by relabeling the sequences if needed, we may assume that
$\gam_k<2\gd$ for all $k\in\N$. 

Fix $k\in\N$ and consider the initial value problem
\[
\dot Y_k(t)=\frac{1}{\ep_k}b(Y_k(t))-\nu F(Y_k(t)) \ \ \ \text{ and } \ \ \ Y_k(0)=z,
\]
where the function $F$ is given by $F(x):=DH(x)/|DH(x)|$. Let $I_k$ denote 
the maximal interval of existence of the solution $Y_k(t)$. Noting that
\[
\frac{d}{dt}H(Y_k(t))=-\nu|DH(Y_k(t))| \ \ \text{ for } t\in I_k,
\] 
we deduce that there exist $\gs_k, \tau_k\in I_k$ such that $\gs_k<0<\tau_k$,
\[
H(Y_k(\gs_k))=\gam_k \ \ \ \text{ and } \ \ \ H(Y_k(\tau_k))=-\gd. 
\]
According to Lemma \ref{L: max9}, there are constants $\ga\in (0,\,1)$ and $A_0>0$ such that 
\[
|DH(x)|\geq A_0|H(x)|^\ga \ \ \text{ for } x\in\gO.
\]
Noting that $Y_k(t)\in\gO$ for $t\in [\gs_k,\,\tau_k]$, we compute that for $t\in(\gs_k,\,0)$,
\[
\frac{d}{dt}H(Y_k(t))^{1-\ga}=(1-\ga)H(Y_k(t))^{-\ga}\frac{d}{dt}H(Y_k(t))\leq -(1-\ga)\nu A_0, 
\]
and, after integration over $(\gs_k,\,0)$, 
\[
-\gam_k^{1-\ga}\leq -(1-\ga)\nu A_0|\gs_k|,
\]
which ensures that 
\begin{equation}\label{E: max41}
-\gs_k=|\gs_k|\leq \frac{\gam_k^{1-\ga}}{(1-\ga)\nu A_0}\leq \frac{(2\gd)^{1-\ga}}{(1-\ga)\nu A_0}. 
\end{equation}
Similarly, we deduce that
\begin{equation}\label{E: max42}
\tau_k\leq \frac{\gd^{1-\ga}}{(1-\ga)\nu A_0}. 
\end{equation}

Since $|F(x)|=1$ and, hence, $u_k^*$ is a viscosity subsolution of 
\[
\gl u_k^*-\left(\frac b{\ep_k}-\nu F\right)\cdot Du_k^*-M=0 \ \ \IN \gO\setminus\{0\}
\]
by (G5), we may apply Proposition \ref{app1}, to obtain
\[
u_k^*(Y_k(\gs_k))\leq e^{\gl\gs_k}\left( e^{-\gl \tau_k}u_k^*(Y_k(\tau_k)) 
+\int_{\ga_k}^{\tau_k} Me^{-\gl t}\,dt\right). 
\]
Recalling that $\gam_k=H(x_k)=H(Y_k(\gs_k))$, we 
combine the above with Lemma \ref{L: loop}, to get
\[\bald
u_k^*(x_k)&\,\leq C\ep_k T_0(\gam_k) 
+e^{\gl(\gs_k-\tau_k)} u_k^*(Y_k(\tau_k)) 
+\gl^{-1}M\left(1-e^{\gl(\gs_k-\tau_k)}\right)
\\&\, \leq C\ep_k T_0(\gam_k)+u_k^*(Y_k(\tau_k))+ 
C_M(1+\gl^{-1})\left(1-e^{\gl(\gs_k-\tau_k)}\right),
\eald
\]
and, moreover, by \eqref{E: max41} and \eqref{E: max42}, 
\[
u_k^*(x_k)\leq C\ep_k T_0(\gam_k)+\max_{c_i(-\gd)}u_k^*+ 
C_M(1+\gl^{-1})\left\{1-\exp\left(-\gl \left(\frac{(2\gd)^{1-\ga}+\gd^{1-\ga}}{(1-\ga)\nu A_0}\right)\right)\right\}.
\]
Sending $k\to\infty$ yields
\[
v_0^+(\gd)\leq v_i^+(-\gd)+
C_M(1+\gl^{-1})\left\{1-\exp\left(-\gl \left(\frac{(2\gd)^{1-\ga}+\gd^{1-\ga}}{(1-\ga)\nu A_0}\right)\right)\right\},
\]
and hence, $d(v_0^+)\leq d(v_i^+)$. 
\end{proof}

\begin{cor} \label{C: max3} For every $i\in\cI_0$,
\[
v^+(x)=v_i^+(0)=d(v_i^+) \ \ \text{ for all } x\in c_0(0). 
\]
In particular, $v_0^+(0)=\ldots=v_{N-1}^+(0)$.
\end{cor}

\begin{proof} Combining Lemmas \ref{L: max1}, \ref{L: max3}, and \ref{L: max4} yields
\[
\max\{d(v_0^+),\,d(v_i^+)\}\leq \inf_{c_i(0)}v^+\leq \max_{c_0(0)}v^+\leq d(v_0^+)
\leq d(v_i^+) \ \ \text{ for all } i\in\cI_1,
\]
which shows that
\[
\inf_{c_i(0)}v^+=\max_{c_0(0)} v^+=d(v_i^+)=d(v_0^+) \ \ \text{ for all } i\in\cI_1.
\]
Since $c_0(0)=\bigcup_{i\in\cI_1}c_i(0)$, we 
conclude that 
\[
v^+(x)=d(v_i^+) \ \ \ \text{ for all } x\in c_0(0), \, i\in\cI_0,
\]  
and, by the definition of $v_i^+(0)$, 
\[
v_i^+(0)=d(v_i^+) \ \ \ \text{ for all } i\in\cI_0. \qedhere
\]
\end{proof}

For the proof of Theorem \ref{T: Cl2}, we argue below as in the proof of \cite[Lemma 3.8]{Ku16}.  We need the following lemma, the proof of which we refer to  
\cite[Lemma 4.4]{Ku16}.

\begin{lem} \label{L: key}
For any $\eta > 0$,
there exist a constant $\gd \in (0, \, \bar h)$ and a function $\psi \in C^1 (\gO(\gd))$ such that
\begin{equation*}
-b \cdot D\psi + G(x, 0) < G(0, 0) + \eta \ \ \  \In \gO(\gd).
\end{equation*}
\end{lem}

\begin{proof}[Proof of Theorem \ref{T: Cl2}] We set $d=d(u_0,\ldots,u_{N-1})$, and note 
by the maximality of $(u_0,\ldots,u_{N-1})$, Corollary \ref{C: max3}, and Theorem \ref{T: Cl1}  that 
\[
v^-(x)\leq v^+(x)=d(v_0^+)=\cdots=d(v_{N-1}^+)\leq d \ \ \text{ for all } x\in c_0(0).
\]
It remains to show that
\begin{equation} \label{E: Cl21}
v^-(x) \geq d \ \ \  \text{ for all } x \in c_0(0).
\end{equation}

To prove \eqref{E: Cl21}, we argue by contradiction, and suppose that $\min_{c_0(0)}v^- < d$. We set $\gk:=\min_{c_0(0)} v^-$.

For any $i\in \cI_0$, we have 
\[
\gl u_i(h)+\min_{q\in\R}\ol G_i(h,q)\leq 0 \ \ \text{ for all } h\in J_i,
\]
and, hence, by Lemma \ref{L: G0}, 
\begin{equation} \label{E: gk}
\gl \gk+\lim_{J_i\ni h\to 0}\ol G_i(h,0)=\gl\gk+G(0,0)<0. 
\end{equation}
Combining this and the equi-continuity (see (ii) of Lemma \ref{L: PG}) of $q\mapsto \ol G_i(h,q)$, with $h\in J_i$, 
we deduce that there exists $\gd>0$ such that for all $i\in\cI_0$ and $h\in [-\gd,\,\gd]\cap J_i$, 
\[
\gl(\gk+\gd^2)+\ol G_i(h,\gd_i)<-\gd \ \ \text{ and } \ \ u_i(h)\geq \gk+\gd^2,
\]
where $\gd_i=\gd$ if $i=0$ and $=-\gd$ otherwise, which implies that 
for all $i\in\cI_0$ and $h\in [-\gd,\,\gd]\cap J_i$, 
\begin{equation}\label{E: exist-gd}
\gl (\gk+\gd_ih) + \ol G_i(h, \gd_i) < -\gd \ \ \  \text{ and } \ \ \  u_i(h) \geq \gk + \gd_ih
\end{equation}
By \eqref{E: gk}, we may assume as well that 
\[
\gl\gk +G(0,0)<-\gd. 
\]

According to Lemma \ref{L: key}, we may choose, after replacing $\gd>0$ by a smaller number if necessary, a function $\psi\in C(\ol\gO(\gd))$ such that
\[
-b\cdot D\psi+G(x,0)<G(0,0)+\gd \ \ \IN \gO(\gd).
\]
This yields 
\begin{equation} \label{E: psi}
\gl\gk-b\cdot D\psi+G(x,0)<0 \ \ \IN \gO(\gd). 
\end{equation}

For each $i \in \cI_0$, we define the function $w_i$ on $\bar J_i$ by
\begin{equation*}
w_i(h) =
\begin{cases}
\gd_ih + \gk \ \ \  &\For h \in \bar J_i \cap [-\gd, \, \gd], \\
u_i(h) - u_i(\gd_i) + \gd^2 + \gk \ \ \  &\For h \in \bar J_i \setminus [-\gd, \, \gd].
\end{cases}
\end{equation*}
By Lemma \ref{L: PS1},  the function $u_i$ is locally Lipschitz continuous in $\bar J_i\setminus\{0\}$ and, hence,
$w_i$ is Lipschitz continuous on $\bar J_i$.
Moreover, thanks to the convexity of $\ol G_i(h,q)$ in $q$, i.e.,  (iii) of Lemmas \ref{L: PG}, the function $w_i$ is a viscosity subsolution of (HJ$_i$) and (BC$_i$).
Note that $v_i^-$ is a viscosity supersolution of (HJ$_i$) and (BC$_i$) and
satisfies $\liminf_{J_i \ni h \to 0} v_i^-(h) \geq \gk$.
Hence, by applying Lemma \ref{L: limCP}, we obtain
\begin{equation*}
w_i(h) \leq v_i^-(h) \ \ \  \text{ for all } h \in J_i \text{ and } i \in \cI_0.
\end{equation*}
Fix any $\mu \in (0, \, \gd^2)$. The inequality above allows us to choose $\ep_0>0$ so that for any $\ep\in(0,\,\ep_0)$ and $u\in\cS_\ep$, 
\begin{equation}\label{E: boundary-ineq}
\gd^2 + \gk - \mu < u_*(x) \ \ \  \text{ for all } x \in \pl\gO(\gd).
\end{equation}

We next choose a constant $a \in (\gk, \, \gd^2 + \gk - \mu)$, define the function $z^\ep$ on $\ol\gO(\gd)$ by 
\[
z^\ep(x)=a+\ep \psi(x),
\]
and compute by \eqref{E: psi} that for any $x\in\gO(\gd)$,
\[\bald
\gl z^\ep(x)&-\frac 1\ep\, b(x)\cdot Dz^\ep(x) +G(x,Dz^\ep(x))
\\&=\gl (a-\gk)+\gl \ep \psi(x)+G(x,\ep D\psi(x))-G(x,0).  
\eald\]
Reselecting $\ep_0>0$ small enough if needed, 
we see that for any $\ep\in(0,\,\ep_0)$, 
the function $z^\ep$ is a viscosity subsolution of \eqref{E: epHJ} 
in $\gO(\gd)$. Moreover, we may assume that for any 
$\ep\in(0,\,\ep_0)$, 
\[
z^\ep(x)\leq \gd^2+\gk-\mu \ \ \ON \ol\gO(\gd). 
\]
Hence, by the comparison principle for \eqref{E: epHJ} on $\ol\gO(\gd)$, we get
\begin{equation*}
z^\ep(x) \leq u_*(x) \ \ \  \text{ for all } u\in\cS_\ep\, \text{ and }\, x \in \ol\gO(\gd),
\end{equation*}
which yields a contradiction:  
\begin{equation*}
\gk<a \leq v^-(x) \ \ \  \text{ for all } x \in c_0(0).
\end{equation*}
This completes the proof. 
\end{proof}


  \renewcommand{\theequation}{A.\arabic{equation}}
 \renewcommand{\thethm}{A.\arabic{thm}}
  \setcounter{equation}{0}  
\setcounter{thm}{0}

\section*{Appendix}

\begin{prop} \label{app1}
Let $m\in\N$ be such that $m\geq 2$, $U$ an open subset of $\R^m$ 
and $E\map U\to\R^m$ a Lipschitz continuous vector field. 
Let $v\in\USC(U)$ be a viscosity subsolution of 
\[
\gl v-E\cdot Dv-f=0 \ \ \IN U,
\]
where $\gl\geq 0$ is a given constant and $f\in C(U)$ be a given function. 
Let $c,d\in\R$ be such that $c<d$ and let $X\map (c,d)\to U$ be a $C^1$-curve 
such that 
\[
\dot X(t)=E(X(t)) \ \ \text{ for all }\ t\in(c,d).  
\]
Set $w(t)=v(X(t))$ and $g(t)=f(X(t))$ for $t\in (c,\,d)$ Let \ $\gs,\tau$\ 
be real numbers such that \ $c<\gs<\tau<d$. Then 
\[
e^{-\gl\gs}w(\gs)\leq e^{-\gl \tau}w(\tau)
+\int_{\gs}^{\tau}e^{-\gl t}g(t)\,dt. 
\]
\end{prop} 

\def\t{\hat t} \def\x{\hat x}

\begin{proof} Set $I=(c,\,d)$. It is obvious that \ $w\in\USC(I)$. 
We show first that $w$ is 
a viscosity subsolution of  
\begin{equation}\label{A0}
\gl w-w'-g=0 \ \ \IN I.
\end{equation}

For this, let $\phi\in C^1(I)$ and assume that $w-\phi$ has a strict 
maximum at $\t\in I$.  Set $\x=X(\t)$ and choose $\gd>0$ so that 
\[
[\t-\gd,\,\t+\gd]\subset I \ \ \text{ and } \ \ \ol B_\gd(\x)\subset U.
\]
Fix any $\ga>0$ and consider the function 
\[
\Phi_\ga(t,x):=v(x)-\phi(t)-\ga|x-X(t)|^2 \ \ \ \ON K:=[\t-\gd,\,\t+\gd]\tim\ol B_\gd(\x).
\]
Let $(t^\ga,\,x^\ga)\in K$ be a maximum point of $\Phi_\ga$. It is easily seen that, as $\ga\to\infty$,
\[
(t^\ga,\,x^\ga)\ \to\ (\t,\x) \ \ \text{ and } \ \ \ga|x^\ga-X(t^\ga)|^2 \ \to \ 0.
\] 
Accordingly, by assuming $\ga$ large enough, we may assume that 
$(t^\ga,\,x^\ga)\in (\t-\gd,\,\t+\gd)\tim B_\gd(\x)$, and, 
by the viscosity property of $v$, we have
\[
\gl v(x^\ga)-E(x^\ga)\cdot 2\ga(x^\ga-X(t^\ga))-f(x^\ga)\leq 0.
\] 
Also,  since $t\mapsto \Phi_\ga(t,x^\ga)$ has a local maximum at $t^\ga$, we have
\[
-\phi'(t^\ga)-2\ga(X(t^\ga)-x^\ga)\cdot \dot X(t^\ga)=0. 
\]
Adding these two yields
\[
\gl v(x^\ga)-\phi'(t^\ga)-2\ga(x^\ga-X(t^\ga))\cdot (E(x^\ga)-E(X(t^\ga)))-f(x^\ga)\leq 0.
\]
Hence, letting $C$ be the Lipschitz constant of the function $E$, we obtain
\[
\gl v(x^\ga)-\phi'(t^\ga)-2\ga C|x^\ga-X(t^\ga)|^2-f(x^\ga)\leq 0.
\]
Sending $\ga\to\infty$ in the above, we get \ $\gl v(X(\t))-\phi'(\t)-f(X(\t))\leq 0$, 
and conclude that $w$ satisfies \eqref{A0} in the viscosity sense.

To complete the proof, we fix any $\tau\in I$. The function 
\[
z(t):=e^{\gl t}\left(e^{-\gl\tau}w(\tau)+\int_t^\tau e^{-\gl s}g(s)\,ds\right) 
\] 
is a classical solution of \eqref{A0} and satisfies the condition that $z(\tau)=w(\tau)$. 
Fix any $\gs\in(c,\,\tau)$, choose $a\in(c,\,\gs)$, and, for $\ep>0$,  set 
\[
\chi_\ep(t)=\frac{\ep}{t-a} \ \ \text{ for } t\in(a,\,\tau]. 
\]
The function $\gz_\ep:=z+\chi_\ep$ on $(a,\,\tau]$ satisfies in the classical sense
\[
\gl \gz_\ep-\gz_\ep'-g>0 \ \ \IN (a,\,\tau) \ \ \text{ and } \ \ \gz_\ep(\tau)>w(\tau). 
\]
If $w-\gz_\ep$ has a maximum at some point in $(a,\,\tau)$, then 
the first inequality above yields a contradiction. On the other hand, since 
$\lim_{t\to a+}(w-\gz_\ep)(t)=-\infty$ and $(w-\gz_\ep)(\tau)<0$, 
the function $w-\gz_\ep$ has a maximum 
at a point $t_0\in(a,\,\tau]$ and, moreover, $t_0=\tau$, which implies that 
\[
(w-\gz_\ep)(t)\leq (w-\gz_\ep)(\tau)<0 \ \ \text{ for all } t\in(a,\,\tau].
\] 
Sending $\ep\to 0$, we see that 
\[
w(\gs)\leq z(\gs)=e^{\gl\gs}\left(e^{-\gl\tau}w(\tau)+\int_\gs^\tau e^{-\gl s}F(s)\,ds\right), 
\]
which finishes the proof. 
\end{proof}

\bye

%% file: domain.tex
{\unitlength 0.1in%
\begin{picture}(31.5100,29.8700)(20.5000,-44.5700)%
%
\special{pn 8}%
\special{pa 3259 3148}%
\special{pa 3227 3153}%
\special{pa 3196 3157}%
\special{pa 3164 3160}%
\special{pa 3068 3163}%
\special{pa 3036 3162}%
\special{pa 3004 3160}%
\special{pa 2940 3154}%
\special{pa 2908 3150}%
\special{pa 2877 3145}%
\special{pa 2845 3139}%
\special{pa 2814 3133}%
\special{pa 2782 3127}%
\special{pa 2720 3115}%
\special{pa 2688 3108}%
\special{pa 2658 3100}%
\special{pa 2627 3091}%
\special{pa 2596 3081}%
\special{pa 2566 3072}%
\special{pa 2535 3062}%
\special{pa 2505 3051}%
\special{pa 2475 3041}%
\special{pa 2445 3029}%
\special{pa 2415 3018}%
\special{pa 2386 3005}%
\special{pa 2357 2991}%
\special{pa 2328 2978}%
\special{pa 2299 2964}%
\special{pa 2271 2948}%
\special{pa 2244 2931}%
\special{pa 2217 2913}%
\special{pa 2165 2877}%
\special{pa 2140 2857}%
\special{pa 2116 2835}%
\special{pa 2095 2811}%
\special{pa 2075 2786}%
\special{pa 2060 2758}%
\special{pa 2052 2727}%
\special{pa 2050 2695}%
\special{pa 2061 2665}%
\special{pa 2081 2640}%
\special{pa 2106 2620}%
\special{pa 2134 2605}%
\special{pa 2164 2593}%
\special{pa 2194 2584}%
\special{pa 2226 2576}%
\special{pa 2257 2571}%
\special{pa 2321 2565}%
\special{pa 2353 2563}%
\special{pa 2385 2562}%
\special{pa 2449 2564}%
\special{pa 2481 2566}%
\special{pa 2513 2569}%
\special{pa 2544 2572}%
\special{pa 2576 2577}%
\special{pa 2640 2585}%
\special{pa 2671 2590}%
\special{pa 2703 2596}%
\special{pa 2765 2610}%
\special{pa 2796 2618}%
\special{pa 2827 2625}%
\special{pa 2858 2633}%
\special{pa 2920 2651}%
\special{pa 2980 2673}%
\special{pa 3010 2685}%
\special{pa 3039 2697}%
\special{pa 3069 2709}%
\special{pa 3098 2721}%
\special{pa 3128 2734}%
\special{pa 3156 2748}%
\special{pa 3185 2763}%
\special{pa 3213 2779}%
\special{pa 3240 2795}%
\special{pa 3268 2812}%
\special{fp}%
%
\special{pn 8}%
\special{pa 3268 2809}%
\special{pa 3625 3084}%
\special{fp}%
%
\special{pn 8}%
\special{pa 3625 3084}%
\special{pa 3250 3150}%
\special{fp}%
%
\special{pn 8}%
\special{pa 3470 2754}%
\special{pa 3458 2724}%
\special{pa 3436 2664}%
\special{pa 3426 2634}%
\special{pa 3416 2603}%
\special{pa 3408 2572}%
\special{pa 3402 2541}%
\special{pa 3397 2509}%
\special{pa 3392 2478}%
\special{pa 3386 2446}%
\special{pa 3381 2415}%
\special{pa 3377 2383}%
\special{pa 3374 2351}%
\special{pa 3372 2319}%
\special{pa 3371 2287}%
\special{pa 3369 2255}%
\special{pa 3369 2223}%
\special{pa 3368 2191}%
\special{pa 3368 2095}%
\special{pa 3369 2063}%
\special{pa 3375 1999}%
\special{pa 3378 1968}%
\special{pa 3386 1904}%
\special{pa 3391 1873}%
\special{pa 3403 1809}%
\special{pa 3409 1778}%
\special{pa 3417 1747}%
\special{pa 3427 1717}%
\special{pa 3445 1655}%
\special{pa 3457 1626}%
\special{pa 3472 1598}%
\special{pa 3488 1570}%
\special{pa 3505 1542}%
\special{pa 3525 1517}%
\special{pa 3549 1496}%
\special{pa 3576 1480}%
\special{pa 3607 1471}%
\special{pa 3639 1473}%
\special{pa 3668 1488}%
\special{pa 3693 1507}%
\special{pa 3715 1530}%
\special{pa 3734 1556}%
\special{pa 3751 1583}%
\special{pa 3765 1612}%
\special{pa 3778 1641}%
\special{pa 3800 1701}%
\special{pa 3810 1732}%
\special{pa 3818 1762}%
\special{pa 3825 1794}%
\special{pa 3839 1856}%
\special{pa 3845 1888}%
\special{pa 3850 1919}%
\special{pa 3856 1983}%
\special{pa 3860 2015}%
\special{pa 3863 2047}%
\special{pa 3866 2078}%
\special{pa 3868 2110}%
\special{pa 3870 2174}%
\special{pa 3869 2206}%
\special{pa 3869 2238}%
\special{pa 3868 2270}%
\special{pa 3866 2302}%
\special{pa 3857 2398}%
\special{pa 3853 2430}%
\special{pa 3849 2461}%
\special{pa 3839 2525}%
\special{pa 3825 2587}%
\special{pa 3817 2618}%
\special{pa 3808 2649}%
\special{pa 3798 2679}%
\special{fp}%
%
\special{pn 8}%
\special{pa 3801 2678}%
\special{pa 3624 3093}%
\special{fp}%
%
\special{pn 8}%
\special{pa 3624 3093}%
\special{pa 3469 2747}%
\special{fp}%
%
\special{pn 8}%
\special{pa 3992 3148}%
\special{pa 4024 3153}%
\special{pa 4055 3157}%
\special{pa 4087 3160}%
\special{pa 4119 3161}%
\special{pa 4151 3163}%
\special{pa 4183 3163}%
\special{pa 4215 3162}%
\special{pa 4247 3160}%
\special{pa 4311 3154}%
\special{pa 4343 3150}%
\special{pa 4374 3145}%
\special{pa 4406 3139}%
\special{pa 4437 3133}%
\special{pa 4469 3127}%
\special{pa 4531 3115}%
\special{pa 4624 3091}%
\special{pa 4655 3081}%
\special{pa 4685 3072}%
\special{pa 4716 3062}%
\special{pa 4806 3029}%
\special{pa 4836 3017}%
\special{pa 4923 2978}%
\special{pa 4952 2964}%
\special{pa 4980 2948}%
\special{pa 5007 2931}%
\special{pa 5034 2913}%
\special{pa 5060 2896}%
\special{pa 5086 2877}%
\special{pa 5111 2857}%
\special{pa 5135 2835}%
\special{pa 5156 2811}%
\special{pa 5176 2786}%
\special{pa 5190 2757}%
\special{pa 5199 2726}%
\special{pa 5200 2694}%
\special{pa 5189 2664}%
\special{pa 5169 2639}%
\special{pa 5145 2619}%
\special{pa 5116 2604}%
\special{pa 5086 2593}%
\special{pa 5056 2584}%
\special{pa 5025 2576}%
\special{pa 4993 2571}%
\special{pa 4929 2565}%
\special{pa 4897 2563}%
\special{pa 4865 2562}%
\special{pa 4801 2564}%
\special{pa 4769 2566}%
\special{pa 4737 2569}%
\special{pa 4706 2573}%
\special{pa 4610 2585}%
\special{pa 4579 2590}%
\special{pa 4547 2596}%
\special{pa 4485 2610}%
\special{pa 4454 2618}%
\special{pa 4423 2625}%
\special{pa 4392 2634}%
\special{pa 4361 2642}%
\special{pa 4330 2651}%
\special{pa 4270 2673}%
\special{pa 4240 2685}%
\special{pa 4211 2697}%
\special{pa 4181 2709}%
\special{pa 4123 2735}%
\special{pa 4065 2763}%
\special{pa 4037 2779}%
\special{pa 4010 2795}%
\special{pa 3983 2812}%
\special{fp}%
%
\special{pn 8}%
\special{pa 3983 2809}%
\special{pa 3625 3084}%
\special{fp}%
%
\special{pn 8}%
\special{pa 3625 3084}%
\special{pa 4000 3150}%
\special{fp}%
%
\special{pn 8}%
\special{pa 3550 3465}%
\special{pa 3536 3527}%
\special{pa 3527 3558}%
\special{pa 3517 3589}%
\special{pa 3508 3619}%
\special{pa 3497 3649}%
\special{pa 3473 3709}%
\special{pa 3460 3738}%
\special{pa 3446 3767}%
\special{pa 3416 3823}%
\special{pa 3399 3851}%
\special{pa 3383 3878}%
\special{pa 3366 3905}%
\special{pa 3349 3933}%
\special{pa 3332 3959}%
\special{pa 3314 3986}%
\special{pa 3295 4012}%
\special{pa 3277 4038}%
\special{pa 3259 4065}%
\special{pa 3240 4090}%
\special{pa 3219 4115}%
\special{pa 3198 4139}%
\special{pa 3178 4163}%
\special{pa 3157 4188}%
\special{pa 3136 4212}%
\special{pa 3114 4235}%
\special{pa 3091 4258}%
\special{pa 3045 4302}%
\special{pa 3021 4323}%
\special{pa 2971 4363}%
\special{pa 2944 4381}%
\special{pa 2918 4399}%
\special{pa 2890 4415}%
\special{pa 2861 4429}%
\special{pa 2832 4442}%
\special{pa 2801 4451}%
\special{pa 2770 4456}%
\special{pa 2738 4454}%
\special{pa 2707 4444}%
\special{pa 2683 4424}%
\special{pa 2667 4396}%
\special{pa 2656 4366}%
\special{pa 2651 4335}%
\special{pa 2651 4303}%
\special{pa 2653 4271}%
\special{pa 2656 4239}%
\special{pa 2662 4207}%
\special{pa 2671 4176}%
\special{pa 2679 4146}%
\special{pa 2688 4115}%
\special{pa 2698 4085}%
\special{pa 2722 4025}%
\special{pa 2734 3996}%
\special{pa 2748 3966}%
\special{pa 2762 3938}%
\special{pa 2777 3910}%
\special{pa 2792 3881}%
\special{pa 2807 3853}%
\special{pa 2823 3825}%
\special{pa 2839 3798}%
\special{pa 2856 3770}%
\special{pa 2873 3744}%
\special{pa 2891 3717}%
\special{pa 2910 3691}%
\special{pa 2928 3665}%
\special{pa 2948 3639}%
\special{pa 2967 3614}%
\special{pa 3007 3564}%
\special{pa 3028 3539}%
\special{pa 3049 3516}%
\special{pa 3071 3492}%
\special{pa 3093 3469}%
\special{pa 3114 3445}%
\special{pa 3137 3423}%
\special{pa 3187 3383}%
\special{pa 3212 3362}%
\special{pa 3236 3342}%
\special{pa 3237 3341}%
\special{fp}%
%
\special{pn 8}%
\special{pa 3234 3341}%
\special{pa 3617 3099}%
\special{fp}%
%
\special{pn 8}%
\special{pa 3617 3099}%
\special{pa 3547 3474}%
\special{fp}%
%
\special{pn 8}%
\special{pa 3688 3465}%
\special{pa 3702 3527}%
\special{pa 3710 3558}%
\special{pa 3720 3589}%
\special{pa 3730 3619}%
\special{pa 3741 3649}%
\special{pa 3765 3709}%
\special{pa 3778 3738}%
\special{pa 3792 3767}%
\special{pa 3806 3795}%
\special{pa 3822 3823}%
\special{pa 3839 3850}%
\special{pa 3855 3878}%
\special{pa 3872 3905}%
\special{pa 3889 3933}%
\special{pa 3906 3959}%
\special{pa 3924 3986}%
\special{pa 3943 4012}%
\special{pa 3979 4064}%
\special{pa 3998 4090}%
\special{pa 4019 4115}%
\special{pa 4039 4139}%
\special{pa 4060 4164}%
\special{pa 4080 4188}%
\special{pa 4101 4212}%
\special{pa 4123 4236}%
\special{pa 4146 4258}%
\special{pa 4169 4281}%
\special{pa 4192 4302}%
\special{pa 4217 4323}%
\special{pa 4242 4343}%
\special{pa 4267 4362}%
\special{pa 4293 4381}%
\special{pa 4320 4399}%
\special{pa 4348 4415}%
\special{pa 4376 4429}%
\special{pa 4406 4442}%
\special{pa 4436 4451}%
\special{pa 4468 4456}%
\special{pa 4500 4454}%
\special{pa 4530 4444}%
\special{pa 4555 4424}%
\special{pa 4571 4396}%
\special{pa 4581 4366}%
\special{pa 4587 4334}%
\special{pa 4586 4302}%
\special{pa 4582 4238}%
\special{pa 4575 4207}%
\special{pa 4566 4176}%
\special{pa 4550 4114}%
\special{pa 4540 4084}%
\special{pa 4528 4054}%
\special{pa 4516 4025}%
\special{pa 4504 3995}%
\special{pa 4476 3937}%
\special{pa 4446 3881}%
\special{pa 4430 3853}%
\special{pa 4415 3825}%
\special{pa 4399 3797}%
\special{pa 4365 3743}%
\special{pa 4347 3717}%
\special{pa 4290 3639}%
\special{pa 4271 3614}%
\special{pa 4251 3588}%
\special{pa 4231 3563}%
\special{pa 4210 3539}%
\special{pa 4166 3493}%
\special{pa 4145 3469}%
\special{pa 4123 3446}%
\special{pa 4100 3423}%
\special{pa 4050 3383}%
\special{pa 4026 3362}%
\special{pa 4001 3342}%
\special{pa 4000 3341}%
\special{fp}%
%
\special{pn 8}%
\special{pa 4003 3341}%
\special{pa 3621 3099}%
\special{fp}%
%
\special{pn 8}%
\special{pa 3621 3099}%
\special{pa 3690 3474}%
\special{fp}%
\put(42.7500,-29.5500){\makebox(0,0)[lb]{$z_2$}}%
%
\special{sh 1.000}%
\special{ia 3624 3090 22 22 0.0000000 6.2831853}%
\special{pn 8}%
\special{ar 3624 3090 22 22 0.0000000 6.2831853}%
\put(35.2500,-18.2500){\makebox(0,0)[lb]{$D_1$}}%
\put(48.1100,-28.0600){\makebox(0,0)[lb]{$D_2$}}%
\put(42.9000,-42.5100){\makebox(0,0)[lb]{$D_3$}}%
\put(27.7800,-42.5100){\makebox(0,0)[lb]{$D_4$}}%
\put(22.0500,-28.0900){\makebox(0,0)[lb]{$D_5$}}%
\put(35.8900,-29.8000){\makebox(0,0)[lb]{$0$}}%
%
\special{sh 1.000}%
\special{ia 3630 2420 22 22 0.0000000 6.2831853}%
\special{pn 8}%
\special{ar 3630 2420 22 22 0.0000000 6.2831853}%
%
\special{sh 1.000}%
\special{ia 4200 2930 22 22 0.0000000 6.2831853}%
\special{pn 8}%
\special{ar 4200 2930 22 22 0.0000000 6.2831853}%
%
\special{sh 1.000}%
\special{ia 3945 3565 22 22 0.0000000 6.2831853}%
\special{pn 8}%
\special{ar 3945 3565 22 22 0.0000000 6.2831853}%
%
\special{sh 1.000}%
\special{ia 3300 3565 22 22 0.0000000 6.2831853}%
\special{pn 8}%
\special{ar 3300 3565 22 22 0.0000000 6.2831853}%
%
\special{sh 1.000}%
\special{ia 3075 2945 22 22 0.0000000 6.2831853}%
\special{pn 8}%
\special{ar 3075 2945 22 22 0.0000000 6.2831853}%
\put(42.8500,-22.8500){\makebox(0,0)[lb]{$D_0$}}%
\put(39.6500,-37.2500){\makebox(0,0)[lb]{$z_3$}}%
\put(31.6000,-37.3000){\makebox(0,0)[lb]{$z_4$}}%
\put(28.6500,-29.6500){\makebox(0,0)[lb]{$z_5$}}%
\put(35.6500,-23.3500){\makebox(0,0)[lb]{$z_1$}}%
\end{picture}}%